\documentclass[final]{siamart220329} 
\usepackage{amsfonts,amsopn,amsmath,amssymb}
\usepackage{mathrsfs}
\usepackage{color}
\usepackage{booktabs}
\usepackage{nicefrac}
\usepackage{multirow}
\usepackage{enumerate}

\usepackage{todonotes}

\title{An extension of the approximate component mode synthesis method to the heterogeneous Helmholtz equation
}
\author{Elena Giammatteo%
\thanks{Department of Applied Mathematics, University of Twente,
P.O. Box 217, 7500 AE Enschede, The Netherlands.
\email{e.giammatteo@utwente.nl}}
\and Alexander Heinlein%
\thanks{Delft Institute of Applied Mathematics, Delft University of Technology, Mekelweg 4, 2628 CD Delft, The Netherlands. \email{a.heinlein@tudelft.nl}}
\and Matthias Schlottbom%
\thanks{Department of Applied Mathematics, University of Twente,
P.O. Box 217, 7500 AE Enschede, The Netherlands.
\email{m.schlottbom@utwente.nl}}
}

\ifpdf
\hypersetup{
  pdftitle={ACMS method for the Helmholtz equation},
  pdfauthor={E. Giammatteo, A. Heinlein, M. Schlottbom}
}
\fi

\headers{ACMS method for the Helmholtz equation}{E. Giammatteo, A. Heinlein, M. Schlottbom}

\def\NN{\mathbb{N}}
\def\RR{\mathbb{R}}

\def\CC{\mathbb{C}}

\def\V{\mathcal{V}}

\def\I{\mathcal{I}}

\newcommand{\norm}[1]{\ensuremath{\Vert #1 \Vert}}
\newcommand{\Oj}{\Omega_j}

\newcommand{\sesH}{\mathcal{C}} 
\newcommand{\sesI}{\mathcal{A}} 
\newcommand{\sesP}{\mathcal{B}} 

\newcommand{\uB}{u_{B}}
\newcommand{\uG}{u_{\Gamma}}
\newcommand{\uh}{u_S}
\newcommand{\uhB}{u_{B,S}}
\newcommand{\uhG}{u_{\Gamma,S}}

\newcommand{\vG}{v_{\Gamma,S}}
\newcommand{\vB}{v_B}

\newcommand{\VG}{V_{\Gamma}}
\newcommand{\VB}{V_B}
\newcommand{\Vh}{V_{S}}
\newcommand{\VhG}{V_{\Gamma,S_\Gamma}}
\newcommand{\VhB}{V_{B,S_B}}

\newcommand{\E}{\mathcal{E}}
\newcommand{\Ve}{\mathcal{V}}

\newcommand{\ExtG}{E^\Gamma} 

\newcommand{\ExtLoc}[1]{E^{#1}}
\newcommand{\ExtE}{E^e}

\newcommand{\HDO}{H^1_D(\Omega)}
\newcommand{\HOj}{H^1_0(\Oj)}
\newcommand{\Hj}{H^1(\Oj)}
\newcommand{\HO}{H^1(\Omega)}

\newcommand{\HtrG}{H^{\myfrac{1}{2}}(\Gamma)}
\newcommand{\HtrE}{H^{\myfrac{1}{2}}(e)}
\newcommand{\HtrELM}{H_{00}^{\myfrac{1}{2}}(e)}
\newcommand{\He}{H^1(e)}
\newcommand{\HOe}{H^1_0(e)}
\newcommand{\HtrBOj}{H^{\myfrac{1}{2}}(\partial\Oj)}

\newcommand{\myfrac}[2]{\nicefrac{#1}{#2}}

\newsiamremark{assumption}{Assumption} 
\crefname{assumption}{Assumption}{Assumption} 
\newsiamremark{remark}{Remark} 
\crefname{remark}{Remark}{Remark} 

\begin{document}

\maketitle

\begin{abstract}
In this work we propose and analyze an extension of the approximate component mode synthesis (ACMS) method to the heterogeneous Helmholtz equation. The ACMS method has originally been introduced by Hetmaniuk and Lehoucq as a multiscale method to solve elliptic partial differential equations. The ACMS method uses a domain decomposition to separate the numerical approximation by splitting the variational problem into two independent parts: local Helmholtz problems and a global interface problem. While the former are naturally local and decoupled such that they can be easily solved in parallel, the latter requires the construction of suitable local basis functions relying on local eigenmodes and suitable extensions. We carry out a full error analysis of this approach focusing on the case where the domain decomposition is kept fixed, but the number of eigenfunctions is increased. 
The theoretical results in this work are supported by numerical experiments verifying algebraic convergence for the method. In certain, practically relevant cases, even exponential convergence for the local Helmholtz problems can be achieved without oversampling.
\end{abstract}

\section{Introduction} \label{sec:1}

In this paper, we propose and analyze a multiscale method for the heterogeneous Helmholtz equation
\begin{alignat}{4}
-{\rm div}(a \nabla u) - \kappa^2 u &= f &\quad &\text{in } \Omega, \label{eq:model_strong}\\
 a \partial_n u - i \omega \beta u &=g  &\quad  &\text{on } \Gamma_R, \label{eq:model_strong_bcR}\\
 u&=0 &\quad &\text{on } \Gamma_D,
 \label{eq:model_strong_bcD}
\end{alignat}
in a plane region $\Omega\subset \RR^2$. The boundary of $\Omega$ is decomposed into sets $\Gamma_D$ and $\Gamma_R$ modeling, respectively, Dirichlet boundary conditions and impedance boundary conditions described by the function $g$. The real-valued function $\beta $ is related to transmission and reflection of the wave described by $u$.
Material properties of the background medium occupying $\Omega$ are described by the coefficient functions $a$ and $c$. 
Denoting the positive angular frequency as $\omega$, the wavenumber is given by $\kappa=\myfrac{\omega}{c}$, and any interior sources are modeled by the function $f$.
Heterogeneous Helmholtz equations have many applications, such as modeling the propagation of light in photonic crystals, where $u$ is related to either the transverse-electric or  the transverse-magnetic field~\cite{Joannopoulos2008MoldingLight}, or seismic imaging~\cite{Bourgeois1991,Wang2011}.

\subsection{Literature overview}

Numerical simulations of the Helmholtz equation are challenging due to the highly oscillatory behavior of the solution, in particular, in the case of a high wavenumber $\kappa$. The differential operator in the equation is indefinite, which may lead to stability issues for standard discretization approaches such as the classical finite element method (FEM). 
In order to resolve the oscillatory behavior of the solution, a wavenumber-dependent mesh size of $\mathcal{O}(1/\kappa)$ is required. 
Moreover, the FEM solution suffers from the pollution error in general, i.e., the ratio of the errors of the Galerkin solution and the best approximation grows with the wavenumber $\kappa$~\cite[Theorem 2.6]{BabuskaIhlenburgSauter1995}, \cite{BabuskaSauter1997}. 
To obtain accurate FEM solutions,
a much smaller mesh size $h$ satisfying $\kappa^3h^2=\mathcal{O}(1)$ has to be employed~\cite{deraemaeker_dispersion_1999}, see also \cite[Theorem~4.5]{GrahamSauter2019} where a similar condition is derived.

In order to overcome the pollution effect, higher-order methods can be employed: in particular, for the $hp$-FEM, quasi-optimality has been proved in \cite{MelenkSauter2010, MelenkSauter2011} under the conditions that the polynomial degree $p$ is at least $O(\log \kappa)$ and that $\myfrac{\kappa h}{p}$ is sufficiently small.
In \cite{ChaumontFreletNicaise2020, LafontaineSpence2022}, similar results have been achieved for the heterogeneous Helmholtz equation with piecewise smooth coefficients; however, to the authors' knowledge, no theory for rough coefficients is available yet. 

Multiscale discretization approaches can be very efficient for problems with heterogeneous or non-smooth coefficients, especially if the variations are on a much smaller scale than the size of the computational domain. 
While for classical higher-order finite element methods, fine-scale coefficient variations have to be resolved by mesh elements, multiscale methods are typically defined on a coarser grid, and fine-scale variations are handled via adapted basis functions; those basis functions are typically computed locally using a fine-scale mesh. Hence, systems resulting from multiscale methods are often smaller by orders of magnitude.

In recent years, many multiscale discretization methods have been developed and applied to the heterogeneous Helmholtz equation; for a review on numerical homogenization techniques, also addressing their application to Helmholtz problems, see~\cite{altmann_numerical_2021}.
The localized orthogonal decomposition (LOD) method, introduced in~\cite{henning2014localized,malqvist_localization_2014}, has successfully been applied to heterogeneous Helmholtz problems with high wavenumbers~\cite{BrownGallistlPeterseim2017, Peterseim2017, PeterseimVerfuerth2020}.
The LOD relies on a partition of the domain with coarse elements of size $H$, and basis functions that are computed using local oversampling domains of size $\ell H$. The approximation error is then bounded by $\mathcal{O}(H+\gamma^\ell)$ for some $0<\gamma<1$ assuming the resolution condition $H\kappa \ll 1$ and $\log(\kappa)/\ell$ bounded \cite[Theorem~5.5]{Peterseim2017}.
More recently, even super-exponential convergence of the localization error could be shown for a variant of the LOD method in~\cite{Freese2021}.

The heterogeneous multiscale method (HMM)~\cite{abdulle_heterogeneous_2012,e_heterogeneous_2003}, which relies on local periodicity of the media, has been extended to Helmholtz problems in~\cite{OhlbergVerfurth2018}. The authors show quasi-optimality in terms of a wave-number dependent quasi-optimality constant.

Other multiscale discretization approaches are multiscale FEM (MsFEM)~\cite{BabuskaCalozOsborn1994,BabuskaOsborn1983,HouWu1997};
see also, e.g.,~\cite{EfendievHou2009} for an overview on MsFEM. 
Nearly exponential error decay has been achieved in~\cite{MaAlberScheichl2023}, using similar ideas as in the multiscale generalized finite element method (Ms-GFEM), first described in~\cite{Efendiev2013}. 
In a recent work, Chen, Hou and Wang~\cite{ChenHouWang2023} extended ideas from elliptic equations~\cite{Hetmaniuk2010,ChenHouWang2021} to the heterogeneous Helmholtz equation. The method in \cite{ChenHouWang2023} decomposes the solution into a microscale part, which is bounded by $\mathcal{O}(H)$, and macroscale part, which is bounded by $\mathcal{O}(H \gamma^m)$ for some $\gamma<1$ and $m$ denoting the number of local basis functions associated to the edges of a coarse partition with mesh size $H$. These local basis functions are computed by local oversampling.

In domain decomposition methods, multiscale discretizations are also used as coarse spaces to construct preconditioners which are robust with respect to heterogeneous coefficients; see, e.g., \cite{aarnes_multiscale_2002,gander_analysis_2015,heinlein_adaptive_2019,heinlein_fully_2022} for applications to scalar elliptic problems, or~\cite{buck_multiscale_2013} for linear elasticity problems.
For the application to heterogeneous Helmholtz problems, see, for instance,~\cite{bootland_geneo_2021,conen_coarse_2014,gong_domain_2020}.

Let us also refer to~\cite{Wang2011} for an optimized parallel implementation of a direct solver based on local Schur complements on a hierarchy of domain decompositions reducing the number of low-rank compression operations compared with other structured direct solvers; see~\cite{Liu2022} for a recent work in the context of elliptic equations.
Finally, multigrid methods for the heterogeneous Helmholtz problem have been analyzed; see, e.g., \cite{Erlangga_2006}.

In this work, we consider an extension of the approximate component mode synthesis (ACMS) method to Helmholtz problems. The ACMS method has been introduced by Hetmaniuk and Lehoucq in~\cite{Hetmaniuk2010} as a multiscale discretization for scalar elliptic problems with heterogeneous coefficient functions. 
It is based on the early component mode synthesis (CMS) method~\cite{CraigBampton1968,Hurty1960}, which uses a decomposition of the global approximation space into independent local subspaces and an interface space. The basis functions are defined as eigenmodes of corresponding generalized eigenvalue problems. The global support of the interface basis functions leads to high computational cost for their construction as well as a dense matrix structure of the interface problem. Hence, the practicability and potential for parallelization of CMS method are rather limited.
Therefore, the ACMS method was introduced in order to improve the CMS approach: instead of global interface modes, functions with local support are employed.
The basis functions incorporate heterogeneities of the model problem; notably, besides edge modes, this framework uses vertex basis functions of MsFEM~\cite{EfendievHou2009,HouWu1997,buck_multiscale_2013} type. 
Since the ACMS discretization uses problem-specific shape functions with local support, it can therefore be considered a special finite element method (SFEM)~\cite{BabuskaCalozOsborn1994}, and it also fits into the framework of the generalized finite element method (GFEM)~\cite{babuska_generalized_2004}.
In contrast to other multiscale approaches, such as those previously mentioned~\cite{ChenHouWang2023,Freese2021,MaAlberScheichl2023}, the ACMS method of \cite{Hetmaniuk2010} is based on a non-overlapping domain decomposition and does not use local oversampling.

The ACMS method has been further investigated in~\cite{Hetmaniuk2014}, where a priori error bounds and an a posteriori error indicator for the method have been derived, and in~\cite{Heinlein2015}, where a parallel implementation in PETSc based on the finite element tearing and interconnecting dual primal (FETI-DP) domain decomposition method as the parallel solver is presented. 
More recently, a robust spectral coarse space for Schwarz domain decomposition methods based on the ACMS discretization has been introduced in~\cite{heinlein_multiscale_2018}; further related works on CMS and ACMS methods include~\cite{Bourquin1990,Bourquin1992,MadureiraSarkis2018}. 
Moreover, to the authors' best knowledge, preliminary but unpublished work on the extension of the ACMS method to wave problems has been carried out by Hetmaniuk and Johnson in the early 2010s. 

As a multiscale discretization, the ACMS method seems to be well suited for approximating heterogeneous Helmholtz problems. 
As we will discuss in this work, since the eigenmodes of the Helmholtz and the Laplacian operator are the same if the wavenumber is constant, similar ACMS basis functions as in the Laplace case are suitable for the Helmholtz equation as well.

\subsection{Our contribution and outline}

In this paper, we extend the ACMS method to heterogeneous Helmholtz problems. In doing so, we slightly modify the construction of the edge modes and the extension operator. 
These modifications are required because the weak formulation of the considered Helmholtz problem does not lead to a symmetric positive definite formulation, contrary to the elliptic case, for which the ACMS method has been developed originally.
At the same time, we adapt the method to Robin boundary conditions \cref{eq:model_strong_bcR}.
Moreover, edge basis functions are constructed by solving one-dimensional eigenvalue problems that are fully localized to single edges, while the edge eigenvalue problems in~\cite{Hetmaniuk2010,Hetmaniuk2014,Heinlein2015} 
solve eigenproblems that involve extensions to adjacent subdomains.

In contrast to \cite{ChenHouWang2023,Hetmaniuk2014}, where the number of subdomains in the domain decomposition is increased to obtain convergence, we perform a complementary error analysis focusing on fixing the domain decomposition and increasing the number of eigenmodes for tackling the local heterogeneities. 
This approach is motivated by applications from wave propagation in (quasi)-periodic media \cite{Joannopoulos2008MoldingLight}, where our construction yields an accurate local model for the propagation of light within one unit cell, i.e., one subdomain in the domain decomposition. 
While the error analysis for the local sub-problems, for which we obtain similar a posteriori error bounds as in \cite{Hetmaniuk2014} is rather simple, the error analysis for the interface problem is slightly more sophisticated.
The analysis is based on the abstract framework of \cite{GrahamSauter2019}, and relies in particular on the smallness of the so-called adjoint approximability constant; see \cref{sec:theory} and, in particular, \cref{eq:adjoint_approx_constant} below.
In contrast to \cite{GrahamSauter2019}, where $H^2(\Omega)$-regularity of the solution of the adjoint problem is required, we require $H^2$-regularity of the solution of the adjoint problem only in the vicinity of the interface of the domain decomposition.
We show that, upon using sufficiently many (edge) eigenmodes, the adjoint approximability constant can be made arbitrarily small, and we provide an explicit scaling relation for the required number of edge modes in terms of frequency $\omega$.
\\
Finally, we present a detailed numerical study of the method using a flexible implementation that, unlike in the numerical results for the ACMS method shown in previous works, also allows for unstructured meshes and domain decompositions. 

The paper is organized as follows: we introduce the necessary preliminaries for defining and analyzing the ACMS method for Helmholtz problems in~\cref{sec:preliminaries}. In particular, we show the variational formulation of the Helmholtz equation, recall the well-posedness results of \cite{GrahamSauter2019}, and illustrate the underlying domain decomposition and function spaces of the ACMS discretization. 
Then, in \cref{sec:acms}, we present our new variant of the ACMS method for Helmholtz problems along with some theoretical properties. The error analysis of the ACMS method is carried out in \cref{sec:theory}.
Finally, we describe our numerical results in~\cref{sec:results} and conclude with some final remarks.

\section{Preliminaries} \label{sec:preliminaries}
In the following, we introduce the functional analytic setting and recall well-posedness results for the heterogeneous Helmholtz equation. Furthermore, we introduce the domain decomposition used in the ACMS method.

\subsection{Variational formulation of the Helmholtz equation}
We denote by $L^2(\Omega)$ the Lebesgue space of square-integrable functions $v:\Omega\to\CC$ with inner product
$$
(u,v)_\Omega = \int_\Omega u\,\overline{v}\,dx \,,
$$
and by $H^1(\Omega)$ the usual Sobolev space of functions in $L^2(\Omega)$ with square-integrable weak derivatives.
Corresponding notation is used for other measurable sets besides $\Omega$. Furthermore, we indicate by $\HDO\subset H^1(\Omega)$ functions with vanishing trace on $\Gamma_D\subset\partial\Omega$. 
Let us introduce the sesquilinear forms $\sesI,\sesH: \HDO \times \HDO \to \CC$ defined by
\begin{align} 
\label{eq:def_A}
\sesI(u,v)&=\int_{\Omega} a \nabla u \cdot \nabla \overline{v} \ dx \,,\\
\label{eq:def_C}
\sesH(u,v)&=\sesI(u,v) - (\kappa^2 u,v)_\Omega-i (\omega \beta u,v)_{\Gamma_R} \,.
\end{align}
Then the weak form of the Helmholtz problem~\cref{eq:model_strong,eq:model_strong_bcR,eq:model_strong_bcD} and its adjoint are
\begin{align}
    &\text{Find } u \in \HDO: \quad \sesH(u,v) = F(v), \quad \text{for all } v \in \HDO, 
    \label{eq:HelmholtzVariational}
                \\
    &\text{Find } z \in \HDO: \quad \sesH(v,z) = G(v), \quad \text{for all } v \in \HDO.
    \label{eq:HelmholtzVariationalAdjoint}
\end{align}
Here, $F:\HDO \to \CC$ and $G:\HDO \to \CC$, defined by
\begin{align} 
    F(v) &= (f,v)_\Omega + (g,v)_{\Gamma_R},
    \label{eq:def_F}
    \\
    G(v) &= (v,f)_{\Omega}+(v,g)_{\Gamma_R},
    \label{eq:def_G}
\end{align}
are antilinear and linear functionals on $\HDO$, respectively.

\subsection{Well-posedness of the Helmholtz equation}\label{sec:well-posedness-Helmholtz}
We recall the well-posed-ness theory for the Helmholtz equation presented in \cite{GrahamSauter2019}. Thus, we deem the following assumptions valid throughout the paper without explicitly mentioning it.
\begin{assumption}\label{ass:parameters}
	(i) $\Omega\subset\RR^2$ is a connected polygonal domain with piecewise $C^2$ boundary with strictly convex angles.
 Its boundary can be decomposed as $\partial\Omega=\overline{\Gamma_R}\cup\overline{\Gamma_D}$ into relatively open disjoint subsets $\Gamma_R,\Gamma_D\subset\partial\Omega$. The angle between a segment in $\Gamma_D$ and $\Gamma_R$ is not $\pi/2$.\\
 	(ii) The source terms satisfy $f \in L^2(\Omega)$ and $g\in H^{\myfrac{1}{2}}(\Gamma_R)$.
		\\
	(iii) The coefficient functions $a,c \in L^{\infty}(\Omega)$ are uniformly positive, i.e., $ a_{\min} \leq a(x) \leq  a_{\max}$ and $ c_{\min} \leq c(x) \leq c_{\max}$ for a.e. $x\in \Omega$, where $a_{\min}, c_{\min}, a_{\max}, c_{\max}$ are positive constants.
		\\
	(iv) $\beta \in L^{\infty} (\Gamma_R)$ with ${\rm meas}({\rm supp}(\beta))>0$ and either $\beta>0$ or $\beta <0$.
\end{assumption}
The analysis employs the norm $\|u\|_\sesP^2=\sesP(u,u)$ induced by the sesquilinear form 
\begin{align} \label{eq:def_sesP}
\sesP(u,v)= \sesI(u,v) + (\kappa^2 u ,v)_{\Omega},
\end{align}
defined for $u,v\in \HDO$.
In~\cite[Thm.~2.4]{GrahamSauter2019}, the following result is proved. We recall parts of the proof to show robustness of the constant $C_{\sesH}$ for high-frequencies $\omega\to\infty$.
\begin{theorem}\label{thm:well-posedness_Helmholtz}
(i) For all $u,v\in\HDO$, it holds that
\begin{equation}\label{eq:boundedness_sesquilinearH}
    |\sesH(u,v)| \leq C_{\sesH} \norm{u}_{\sesP} \norm{v}_{\sesP},
    \end{equation}
    with constant $C_{\sesH}=1+C_\Omega\beta_{\max} \max\{a_{\min}^{-1},\frac{1+\omega}{\omega}c_{\max}\}$ and constant $C_\Omega$ depending only on $\Omega$.\\
(ii) There exist unique solutions of~\cref{eq:HelmholtzVariational,eq:HelmholtzVariationalAdjoint}.
\\
(iii)
    If $g=0$ in \cref{eq:HelmholtzVariational,eq:HelmholtzVariationalAdjoint}, then there exists a constant $C_{\rm stab} = C_{\rm stab}(a,c,\omega, \Omega)$ such that the corresponding solutions $u$ and $z$ satisfy:
    \begin{equation} \label{eq:Cstab}
    \norm{u}_{\sesP} \leq C_{\rm stab}   \norm{f}_{L^2(\Omega)}, \quad  \norm{z}_{\sesP} \leq C_{\rm stab} \norm{f}_{L^2(\Omega)}.
    \end{equation}
\end{theorem}
\begin{proof} We only derive the constant $C_{\sesH}$, for the other statements see \cite[Thm.~2.4]{GrahamSauter2019}. An application of the Cauchy-Schwarz inequality to \cref{eq:def_C} yields that
    \begin{align*}
        |\sesH(u,v)|&\leq \norm{u}_{\sesP} \norm{v}_{\sesP} +(\omega |\beta| u,u)_{\Gamma_R}^{1/2}(\omega |\beta| v,v)_{\Gamma_R}^{1/2}.
    \end{align*}
    The trace inequality yields a constant $C_\Omega>0$ depending on $\Omega$ such that \cite[p.~41]{Grisvard2011}
    \begin{align*}
        (\omega |\beta| u,u)_{\Gamma_R}&\leq C_\Omega\beta_{\max} \left(\omega \|u\|_{L^2(\Omega)}\|\nabla u\|_{L^2(\Omega)}+ \omega \|u\|_{L^2(\Omega)}^2\right)\\
        &\leq C_\Omega\beta_{\max} \left(\frac{c_{\max}}{\sqrt{a_{\min}}} \|\frac{\omega}{c} u\|_{L^2(\Omega)}\|\sqrt{a}\nabla u\|_{L^2(\Omega)}+ \frac{c_{\max}^2}{\omega} \|\frac{\omega}{c} u\|_{L^2(\Omega)}^2\right)\\
        &\leq C_\Omega\beta_{\max} \max\{\frac{1}{a_{\min}}, \frac{1+\omega}{\omega}c_{\max}^2\}\norm{u}_{\sesP}^2,
    \end{align*}
    from which \cref{eq:boundedness_sesquilinearH} follows.
\end{proof}
Given \Cref{ass:parameters}, it follows that the solution $u$ of \cref{eq:HelmholtzVariational} satisfies $\nabla u \in L^p(\Omega)$ for some $p>2$ \cite{Groeger89}. By invoking the Sobolev embedding theorem \cite[p.~97]{Adams75}, we thus have that $u$ is H\"older continuous in $\bar \Omega$ with exponent $1-2/p$. Condition (i) in \Cref{ass:parameters} is used in \Cref{sec:4.4} below to verify higher regularity of the solution of the dual problem, see \cite{Shamir1968} for a counterexample if (i) is not satisfied.

\subsection{Decomposition of the  computational domain}
Let $\{\Oj\}_{j=1}^J$ denote a conforming decomposition of $\Omega$ into $J$ non-overlapping domains $\Oj$ with piecewise smooth boundaries. 
Furthermore, let 
$$
	\Gamma = \bigcup_{j=1}^J \partial\Oj \setminus \Gamma_D
$$ 
denote the domain decomposition interface and $\E$ the set of all edges of the domain decomposition, where each edge $e$ is an open set with either $\overline{e} = {\partial\Omega_i} \cap {\partial\Oj}$ for some $i \neq j$ or $e\subset\Gamma_R\cap \partial\Omega_i$ for some $i$.
 Furthermore, let $\V=\{p\in\Gamma\}=\Gamma\setminus\cup_{e\in\E}e$ be the set of points connecting adjacent edges, which we also refer to as the vertices of the domain decomposition.

\subsection{Function spaces}\label{sec:function_spaces}
The Lions--Magenes space $\HtrELM$ is defined as an interpolation space between $L^2(e)$ and $\HOe$ \cite[Ch.~1, Thm.~11.7]{LionsMagenes1972}. Therefore, it holds that $\HOe\subset \HtrELM$ densely \cite[p.~10]{LionsMagenes1972}.
We have the following interpolation inequality 
\begin{align}\label{eq:interpolation}
\|\eta \|_{\HtrELM}\leq C \|\eta\|_{L^2(e)}^{\myfrac{1}{2}} \|\eta\|_{\HOe}^{\myfrac{1}{2}}
\end{align}
for all $\eta \in \HOe$ \cite[Proposition 23, p. 19]{LionsMagenes1972}.
The space $\HtrE$ is defined as the interpolation space between $\He$ and $L^2(e)$; see \cite[Ch.~1, Thm.~9.6.]{LionsMagenes1972}.
%
In view of \cite[Ch.~1, Thm.~11.7]{LionsMagenes1972}, if $e\subset\partial\Oj$ for some $1\leq j\leq J$, 
functions in $\HtrELM$ can be extended continuously by zero to functions in $\HtrBOj$. 
Here, $\HtrBOj$ denotes the space of traces of functions in $\Hj$; cf.~\cite[Ch.~1, Thm.~8.3]{LionsMagenes1972}. 
Lastly, by $\HtrG$ we denote the space of traces on $\Gamma$, i.e., $v\in \HtrG$ if, for all $1\leq j\leq J$, there exist $u_j \in \Hj$ such that $u_{j\mid \partial\Oj\cap\Gamma}=v_{\mid \partial\Oj}$.
\section{ACMS method} \label{sec:acms}
The ACMS method introduced in~\cite{Hetmaniuk2010} relies on an orthogonal splitting of $\HO$ into interface functions $\HtrG $ and local functions $H^1_0(\Oj)$, $j=1,\ldots,J$, and on the availability of basis functions with local support; while functions in $H^1_0(\Oj)$ generally have local support, interface basis functions with local support are constructed based on the edges and vertices that form $\Gamma$. In \cite{Hetmaniuk2010}, orthogonality is characterized by the bilinear form associated to the elliptic problem. The basic idea of our extension of the ACMS method to Helmholtz problems is similar,  but `orthogonality' is measured with respect to the sequilinear form $\sesH$ defined in \cref{eq:def_C}, which is not an inner product in general. Below we will construct spaces 
\begin{equation} \label{def:Vapprox}
    \Vh := \VhB \oplus \VhG
\end{equation}
that are either associated to the subdomains $\Oj$ or to the interface $\Gamma$.
In particular, the corresponding basis functions have local support and can be built locally.
We now discuss the construction of the bubble space $\VhB$ and the interface space $\VhG$ in detail.

\subsection{Bubble space}\label{sec:bubble_space}
Let us define the local sesquilinear form $\sesI_{j}:\Hj\times \Hj\to \CC$ as in \cref{eq:def_A} but with domain of integration $\Oj$ instead of $\Omega$.
Since $\sesI_{j}$ is Hermitian, we can consider the eigenproblems: for $j=1,...,J$ and $i \in \NN$, find $(b_i^j,\lambda_i^j)\in \HOj \times \RR$ such that
\begin{align}\label{eq:bubbles}
	\sesI_{j}(b_i^j,v)= \lambda_i^j ( {\kappa^2} b_i^j,v)_{\Oj} \quad \text{for all } v \in \HOj.
\end{align}
Standard theory ensures that the eigenfunctions $\{b^i_j\}_i$ form an orthogonal basis for $\HOj$ with respect to $\sesI_{j}$ and an orthonormal basis for $L^2(\Oj)$ with weighted inner product $(\kappa^2 u,v)_{\Oj}$, and that $\lambda_i^j>0$. Furthermore, we may assume that the eigenvalues $\lambda_i^j$ are ordered non-decreasingly, i.e., $\lambda_i^j\leq \lambda_l^j$ for $i\leq l$.
By definition $\kappa(x)=\omega/c(x)$, and, hence, the numbers $\lambda_i^j\omega^2$ are eigenvalues of a corresponding eigenproblem that is independent of $\omega$. As such, $\lambda_i^j\omega^2$ is independent of $\omega$.
If $c(x)$ is constant, the bubble functions are defined as in the elliptic case \cite{Hetmaniuk2010}.
In slight abuse of notation, we may denote by $b_i^j \in H^1_0(\Omega)$ also the extension of $b_i^j \in \HOj$ by zero outside of $\Oj$ and we call these \textit{bubble functions}.
Associated to the partition $\{\Oj\}_{j}$ of $\Omega$, let us introduce the infinite-dimensional bubble space
\begin{align}\label{eq:bubble_space}
\VB = \bigoplus_{j=1}^{J}V^j,\qquad\text{with } V^j={\rm span}\{ \, b_i^j\,:\, i\in\NN\}.
\end{align}
Let $S_B=(I_1,\ldots, I_J)\in\NN^J$ be a multi-index. Then, the finite-dimensional bubble space employed in the ACMS method is defined by
\begin{align}\label{eq:bubble_space_finite}
\VhB = \bigoplus_{j=1}^{J}V^j_{I_j},\qquad\text{with } V^j_{I_j}={\rm span}\{\, b_i^j \, :\, 1 \leq i \leq I_j \}.
\end{align}

\subsection{Solvability of local Helmholtz problems}
The interface space $\VhG$ introduced below relies on the proper extension of functions defined on the interface $\Gamma$. The extension relies on the solvability of local Helmholtz problems with homogeneous Dirichlet boundary conditions. Therefore, in the rest of the manuscript, we additionally assume the following:
\begin{assumption}\label{ass:eig_not_zero}
    For all $j=1,\ldots,J$ and for all $i\in\NN$, let $\lambda_i^j\neq 1$.
\end{assumption}
\cref{ass:eig_not_zero} might be justified by suitably adapting the partition $\{\Oj\}_j$ such that the spectrum of the corresponding differential operator in \cref{eq:bubble} does not include the value $1$: the eigenvalues $\lambda_i^j$ depend on the size of $\Oj$, and they grow if $\Oj$ is suitably made smaller, cf. \cref{eq:edge_eigs_behavior} for the corresponding scaling behavior of eigenvalues for one-dimensional eigenvalue problems. 
As noted in \Cref{sec:bubble_space}, the numbers $\lambda_i^j\omega^2$ are independent of $\omega$. Hence, the suitable adjustment of $\{\Oj\}$ that guarantees that \cref{ass:eig_not_zero} holds, depends on the frequency $\omega$.
Similar assumptions have been used, e.g., in \cite[Assumption 4.2]{Freese2021} or \cite[Assumption 1]{ChenHouWang2023}, to guarantee coercivity of the local Helmholtz problems.
Next, we establish well-posedness of the local Helmholtz problems under the conditions of \Cref{ass:eig_not_zero}.
\begin{lemma}\label{lem:inf-sup}
	For $\beta^j=\inf_{i\in\NN} \{ \myfrac{|\lambda_i^j-1|}{(\lambda_i^j+1)} \}>0$, the following estimates hold:
	\begin{align*}
		\inf_{u\in \HOj} \sup_{v\in \HOj} \frac{\sesI_{j} (u, v)-(\kappa^2 u, v)_{\Oj}}{\|u\|_{\sesP}\|v\|_{\sesP}} &\geq \beta^j ,\\
		\inf_{v\in \HOj} \sup_{u\in \HOj} \frac{\sesI_{j} (u, v)-(\kappa^2 u, v)_{\Oj}}{\|u\|_{\sesP}\|v\|_{\sesP}} &\geq \beta^j .
	\end{align*}
\end{lemma}
\begin{proof}
    ~\cref{ass:eig_not_zero} implies that $\beta^j>0$.
	Let $u\in \HOj$ be given. Since $\{b_i^j\}_i$ is a basis of $\HOj$, we have that
		$u= \sum_{i=1}^\infty u_i b_i^j$
	with $u_i=(\kappa^2 u, b_i^j)_{\Oj}$. Now, let $v=\sum_{i=1}^\infty {\rm sgn}(\lambda_i^j-1)u_i b_i^j$. 
	We note that 
	$$
	 \|u\|_{\sesP}^2=\sum_{i=1}^\infty (\lambda_i^j+1)|u_i|^2,
	$$
	and $\|v\|_{\sesP}=\|u\|_{\sesP}$.
	Then we have that 
	$$
	\sesI_{j} (u, v)-(\kappa^2 u, v)_{\Oj}  = \sum_{i=1}^\infty |\lambda_i^j-1| u_i^2 \geq \inf_{i\in\NN}\frac{|\lambda_i^j-1|}{\lambda_i^j+1} \|u\|_{\sesP}^2.
	$$
	Hence, the first inequality follows. The second inequality follows similarly.
\end{proof}

\subsection{A harmonic extension operator}
In order to define the interface space, we need an extension from $\Gamma$ to $\Omega$, which is obtained by combining the extensions of functions from $\partial\Oj$ to $\Oj$.
\\
For a given $\tau \in \HtrBOj$, let $\tilde \tau \in \Hj$ be any function satisfying $\tilde\tau_{\mid\partial\Oj}=\tau$.
Then, we indicate by $\tilde\tau_0\in \HOj$ the solution to
\begin{align} \label{eq:tau_definition}
    \sesI_{j} (\tilde\tau_0, v)-(\kappa^2 \tilde\tau_0, v)_{\Oj} = -\left(\sesI_j(\tilde\tau,v)-(\kappa^2 \tilde\tau, v)_{\Oj}\right) \quad\text{for all } v\in  \HOj,
\end{align}
which is uniquely defined by~\cref{lem:inf-sup}.
%
We introduce the local Helmholtz-harmonic extension $\ExtLoc{j}:\HtrBOj\to \Hj$ by setting $\ExtLoc{j}\tau = \tilde\tau+\tilde\tau_0$.
\begin{lemma} \label{lem:extension_bound}
    The extension operator $\ExtLoc{j}:\HtrBOj\to \Hj$ 
    is bounded, that is,
    \begin{equation} \label{eq:extension_bound}
        \|\ExtLoc{j}\tau\|_{\sesP}\leq (1+\myfrac{1}{\beta^j})\|\tilde\tau\|_{\sesP},
    \end{equation}
    where $\tilde \tau\in \Hj$ is any extension of $\tau\in \ \HtrBOj$. 
 \end{lemma} 
\begin{proof}
	We estimate the right-hand side in \cref{eq:tau_definition} by $\|\tilde \tau\|_{\sesP}\|v\|_{\sesP}$.
    Then, \cref{lem:inf-sup} yields that $\|\tilde\tau_0\|_{\sesP} \leq  \myfrac{\|\tilde\tau\|_{\sesP}}{\beta^j}$, and the assertion follows.
\end{proof}
We note the following orthogonality relation, which is a crucial property for the construction of the ACMS  spaces: for $\tau \in \HtrBOj$ and for all bubble functions $b_i^j \in \HOj$, we have
\begin{align}\label{eq:orthogonality_j}
		\sesI_{j}(\ExtLoc{j} \tau,b_i^j)-(\kappa^2\ExtLoc{j} \tau,b_i^j)_{\Oj}=0.
\end{align}
Next, we construct extensions from the local edges and the interface.
We assume that $e=\partial\Oj\cap\partial\Omega_i\in\E$ is a common edge of $\Omega_i$ and $\Oj$. Let $\tau \in \HtrG$, which, by restriction, implies $\tau\in \HtrBOj$ and $\tau\in H^{\myfrac{1}{2}}(\partial \Omega_i)$.
Since $(\ExtLoc{j}\tau)_{\mid e}=\tau_{\mid e}= (\ExtLoc{i}\tau)_{\mid e}$, we can introduce the  extension operator $\ExtG:\HtrG\to \HDO$ by $(\ExtG \tau)_{\mid\Oj} = \ExtLoc{j} \tau_{\mid \partial \Oj}$, for all $j = 1,\ldots,J$. Moreover, we define $\ExtE:\HtrELM\to \HDO$ via $\ExtE\tau = \ExtG E_0^e\tau$, where $E_0^e: \HtrELM  \to \HtrG$ denotes the extension by zero to the interface $\Gamma$.

\subsection{Vertex based approximation space}
For any $p\in\V$, let $\varphi_p: \Gamma\to \RR$ denote a piecewise harmonic function, that is, $\Delta_e\varphi_{p\mid e}=0$ for all $e\in\E$, with $\Delta_e$ indicating the Laplace operator along the edge $e\in\E$, and $\varphi_p(q)=\delta_{p,q}$ for all $p,q\in\V$.
Note that the support of $\varphi_p$ consists of all edges which share the vertex $p$ and is, therefore, local. In turn, $\ExtG\varphi_p$ is supported on all subdomains $\Omega_j$ that share the vertex $p$.
If all edges $e\in\E$ are straight line segments, then $\varphi_p$ is a piecewise linear function on $\Gamma$, similar to \cite{Hetmaniuk2010}.
The choice of $\Delta_e$ here is motivated by the error analysis in \Cref{sec:Interface_approximation_error}.
The vertex based space is then defined by linear combinations of corresponding  extensions,
\begin{align*}
V_{\Ve} = {\rm span}\{\, \ExtG\varphi_p \,:\, \ p\in\Ve\}.
\end{align*}
For our error analysis, we will employ the nodal interpolant
\begin{align}\label{eq:vertex_interpolant}
I_\V v = \sum_{p\in \V} v(p) \varphi_p,
\end{align}
which is well-defined for functions $v:\overline{\Omega}\to\CC$ that are continuous in all $p\in \V$.
Note that $I_\V$ can be applied to the solution $u \in \HDO\cap W^{1,p}(\Omega)$ to \cref{eq:HelmholtzVariational}, see the comments at the end of \Cref{sec:well-posedness-Helmholtz}.

\subsection{Interface space}\label{sec:3.5}
Let us consider $e\in\E$ and denote by $\partial_e$ the tangential derivative, i.e., differentiation along $e$.
We define the edge modes as solutions to the following weak formulation of the edge-Laplace eigenvalue problems: for each $e\in\E$ find $(\tau^e_i,\lambda^e_i)\in \HOe\times\RR$, $i \in \NN$, such that
\begin{align}\label{eq:def_edge_mode}
(\partial_e \tau^e_i,\partial_e \eta)_e =\lambda^e_i ( \tau^e_i, \eta)_e \quad \text{for all } \eta\in \HOe.
\end{align}
Standard theory ensures that the eigenfunctions $\{\tau^e_i\}_i$ form an orthogonal basis for $\HOe$ and an orthonormal basis for $L^2(e)$. We may again assume that the eigenvalues $\lambda_j^e>0$ are ordered increasingly. Moreover, we note that the asymptotic behavior of the eigenvalues is \cite[p.~415]{Courant_Hilbert:1953}
\begin{equation}\label{eq:edge_eigs_behavior}
    \lambda_i^e \sim \Big( \frac{i \pi}{|e|} \Big)^2.
\end{equation}
Let us mention that the definition of the edge modes here involves the resolution of local problems on single edges and it differs from the definition in the classical ACMS method of \cite{Hetmaniuk2010,Hetmaniuk2014}, where the authors solve eigenvalue problems involving the extension operator.
The corresponding infinite-dimensional interface space is
\begin{align} \label{eq:interface_space_infinite}
\VG = V_\V + \sum_{e\in\E } \ExtE V^e,\quad V^e = {\rm span}\{ \tau_i^e\,:\, i \in \NN \}.
\end{align}
Note that each function in $\ExtE V^e$ has local support inside the two subdomains adjacent to $e$. Choosing $I_e\in\NN$, $e\in\E$, we introduce the $L^2(e)$-projection $P^e_{I_e}:L^2(e)\to V^e$ defined by
\begin{align}\label{eq:l2_projection_edge}
P^e_{I_e}v=\sum_{j=1}^{I_e} (v,\tau_j^e)_e \tau_j^e,
\end{align}
and denote the range of $P^e_{I_e}$ by $V^e_{I_e}$.
Collecting the indices $I_e$ in a multi-index $S_\Gamma$,
we define the finite-dimensional interface space by
 \begin{align} \label{eq:interface_space_finite}
    \VhG = V_{\Ve} + \sum_{e\in\E} \ExtE V^e_{I_e},
\end{align}
which, together with \cref{eq:bubble_space_finite}, completes the construction of the ACMS approximation space in \cref{def:Vapprox}.
We now give some well-known interpolation error estimates.
We recall the proofs for convenience of the reader.

\begin{lemma}\label{lem:approximation_properties_edge}
For any $e\in\E$ and all $w\in \HOe$, it holds that
\begin{align}
\|w-P^e_{I_e}w\|_{L^2(e)}
\leq \frac{1}{\sqrt{\lambda_{I_e+1}}} \|w-P^e_{I_e}w\|_{\He}.\label{eq:approximation_err_L2_edge}
\end{align}
If, additionally, $w \in \HOe\cap H^2(e)$, then there exists a constant $C>0$ such that
\begin{align}
\|w- P^e_{I_e}w\|_{\He}
&\leq \frac{C}{\sqrt{\lambda_{I_e+1}^e}} \|\Delta_e w-P^e_{I_e}\Delta_e w\|_{L^2(e)}.\label{eq:approximation_err_H1_edge}
\end{align}
\end{lemma}
\begin{proof}
Since $\{\tau_j^e\}_j$ form an orthonormal basis of $L^2(e)$, we have that
\begin{align*}
\|w-P^e_{I_e}w\|_{L^2(e)}^2 
= \sum_{j\geq I_e+1}|(w,\tau_j^e)_e|^2\leq \frac{1}{\lambda_{I_e+1}} \|w-P^e_{I_e}w\|_{\He}^2,
\end{align*}
which proves \cref{eq:approximation_err_L2_edge}.
By \cref{eq:def_edge_mode} 
and the Poincar\'e inequality, we similarly have that
\begin{align*}
\|w-P^e_{I_e}w\|_{\He}^2 
\leq C_P \sum_{j\geq I_e+1} \lambda_j^e |(w,\tau_j^e)_e|^2.
\end{align*}
 Using the definition of $\tau_j^e$ in \cref{eq:def_edge_mode} again and performing integration by parts, we obtain
\begin{align*}
(w,\tau_j^e)_e = \frac{1}{\lambda_j^e} (\partial_e w,\partial_e \tau_j^e)_e = -\frac{1}{\lambda_j^e} (\Delta_e w,\tau_j^e)_e,
\end{align*}
where we used that $w, \tau_j^e \in \HOe$. This concludes the proof of \cref{eq:approximation_err_H1_edge}.
\end{proof}

\begin{lemma}\label{lem:approximation_properties_edge2}
Let $e\in\E$. Then, there exists a constant $C>0$ such that
\begin{align}\label{eq:interp_1/2_3/2}
\inf_{v^e\in V_{I_e}^e}\| w -I_\V w -v^e\|_{\HtrELM} 
&\leq \frac{C}{\sqrt{\lambda^e_{I_e+1}}} \|w\|_{H^{\myfrac{3}{2}}(e)}
\end{align}
for all $w\in H^{\myfrac{3}{2}}(e)$, with nodal interpolant $I_\V w$ defined in \cref{eq:vertex_interpolant}.
\end{lemma}
\begin{proof}
By continuity of the embedding $\He\hookrightarrow C^0(\overline{e})$, we have that the nodal interpolation operator $I_\V$ is bounded, i.e., $\|I_\V w\|_{\He}\leq C\| w\|_{\He}$ for all $w\in \He$. 
By choosing $v^e=0$, we thus have the stability estimate
\begin{align}\label{eq:stability_adj_appro}
\inf_{v^e\in V_{I_e}^e}\| w -I_\V w -v^e\|_{\He} \leq C \|w\|_{\He}\quad\text{for all } w\in \He.
\end{align}
Employing \cref{lem:approximation_properties_edge} and using that $w-\I_\V w\in \HOe$, we also have that
\begin{align*}
\inf_{v^e\in V_{I_e}^e}\| w -I_\V w -v^e\|_{\He} 
\leq \frac{C}{\sqrt{\lambda_{I_e+1}^e}} \|w\|_{H^2(e)} \quad\text{for all } w\in H^2(e).
\end{align*}
Therefore, we get by interpolation \cite[Thm.~5.1, Thm.~9.6]{LionsMagenes1972} that
\begin{align}\label{eq:interp_1_3/2}
\inf_{v^e\in V_{I_e}^e}\| w -I_\V w -v^e\|_{\He} 
\leq \frac{C}{(\lambda_{I_e+1}^e)^{\myfrac{1}{4}}}\|w\|_{H^{\myfrac{3}{2}}(e)}\quad\text{for all } w\in H^{\myfrac{3}{2}}(e).
\end{align}
Employing \cref{eq:interp_1_3/2} and \cref{lem:approximation_properties_edge} and recalling that the best-approximation is realized via the projection $P^e_{I_e}(w-I_\V w)$, we also have that
\begin{align}\label{eq:interp_0_3/2}
\inf_{v^e\in V_{I_e}^e}\| w -I_\V w -v^e\|_{L^2(e)} 
&\leq \frac{1}{(\lambda^e_{I_e+1})^{\myfrac{3}{4}}} \|w\|_{H^{\myfrac{3}{2}}(e)}\quad\text{for all } w\in H^{\myfrac{3}{2}}(e).
\end{align}
Therefore, by the interpolation inequality \cref{eq:interpolation}, the relations \cref{eq:interp_1_3/2,eq:interp_0_3/2} yield \cref{eq:interp_1/2_3/2}, which concludes the proof.
\end{proof}
The next result shows that $\HDO$-functions can be approximated by bubble and interface functions. The proof of this statement is also the recipe to obtain quantitative error estimates in \Cref{sec:Interface_approximation_error}.
\begin{lemma}
    Let $\VB$ and $\VG$ be as in \cref{eq:bubble_space} and \cref{eq:interface_space_infinite}, respectively.
    Then it holds that $\HDO=\VB\oplus\VG$.
\end{lemma}
\begin{proof} 
Using a density argument, it is sufficient to show that any function $v\in \HDO\cap C^\infty_0(\Omega\cup\Gamma_R)$ can be approximated by functions in $\VB\oplus\VG$.
First note that $(v-\ExtG v_{\mid\Gamma})_{\mid\Oj}\in H^1_0(\Oj)$ for each $j=1,...,J$. Therefore, $v-\ExtG v_{\mid\Gamma}\in \VB$. 
It remains to show that $\ExtG v_{\mid\Gamma}$ can be approximated by functions in the interface space $\VG$.
By continuity of $\ExtG$, it is sufficient to approximate $v_{\mid\partial\Omega_j}$ in $H^{1/2}(\partial\Omega_j)$.
First, we subtract the nodal interpolant, and observe that
$(v-I_\V v)_{\mid e}\in H^1_0(e)\subset H^{1/2}_{00}(e)$ for any $e\in\E$ such that $e\subset\partial\Omega_j$.
Then, $\|v-I_\V v\|_{H^{1/2}(\partial\Oj)}$ can be localized to single edges $e\subset\partial\Oj$ as follows: define $w^e\in H^{1/2}(\partial\Oj)$ via $w^e=(v-I_\V v)_{\mid e}$ and $w^e=0$ on $\partial\Oj\setminus e$. 
We observe that $v-I_\V v=\sum_{e\subset\partial\Omega_j} w^e$ on $\partial\Omega_j$.
From \cite[Lemma~1.3.2.6]{Grisvard2011}, we have that $\|w^e\|_{H^{1/2}(\partial\Oj)}$ is equivalent to $\|w^e\|_{H^{1/2}_{00}(e)}$.
Since it holds $\|v-I_\V v\|_{H^{1/2}(\partial\Oj)}\leq \sum_{e\subset\partial\Oj}\|w^e\|_{H^{1/2}(\partial\Oj)}\leq C \sum_{e\subset\partial\Oj}\|w^e\|_{H^{1/2}_{00}(e)}$, the result follows from noting that $w^e_{\mid e}\in  V^e$.
\end{proof}

\section{Galerkin approximation} \label{sec:theory}

We start by observing that 
\begin{align}\label{eq:ortho}
\sesH(\vB, v_\Gamma) = \sesI(\vB,v_\Gamma) -(\kappa^2 \vB,v_\Gamma)- i (\omega\beta \vB,v_\Gamma)_{\Gamma_R} = 0,
\end{align}
for all $\vB\in \VB$ and $v_\Gamma\in \VG$,
where we used \cref{eq:orthogonality_j} and that $\vB=0$ on $\Gamma_R$.
Since $u=\uB+\uG\in \VB\oplus \VG$ with $\uB=u-\ExtG u_{\mid\Gamma}$ and $\uG=\ExtG u_{\mid\Gamma}$,
we thus obtain that
\begin{align}
\sesH(\uB,\vB) &= F(\vB) \quad \text{for all } \vB \in \VB \quad\text{and} \label{eq:bubble}\\
\sesH(u_\Gamma,v_\Gamma) &= F(v_\Gamma)\quad \text{for all } v_\Gamma \in \VG.\label{eq:interface}
\end{align}

Recalling that $\Vh = \VhB \oplus \VhG$ has been defined in \cref{def:Vapprox}, with $\VhB\subset\VB$ and $\VhG\subset\VG$ defined in \cref{eq:bubble_space_finite} and \cref{eq:interface_space_finite}, respectively, the Galerkin approximations of~\cref{eq:bubble}--\cref{eq:interface}
are split into two independent problems, namely a \textit{bubble approximation}: find $\uhB\in V_{B,S_B}$ such that
\begin{align}\label{eq:bubble_approx}
\sesH(\uhB,v_{B,S}) = F(v_{B,S}) \quad \text{for all } v_{B,S} \in V_{B,S_B};
\end{align}
and an \textit{interface approximation}: find $\uhG\in \VhG$ such that
\begin{align}\label{eq:interface_approx}
	\sesH(u_{S,\Gamma},v_{S,\Gamma}) = F(v_{S,\Gamma})\quad \text{for all } v_{S,\Gamma} \in \VhG.
\end{align}
The Galerkin problem \cref{eq:bubble_approx} is well-posed under \cref{ass:eig_not_zero}, while adjoint approximability is required to analyze the Galerkin problem \cref{eq:interface_approx}; cf. \cite{GrahamSauter2019}. In \cref{thm:estimate_sigma} below, we prove estimates for the adjoint approximability constant, ensuring that also \cref{eq:interface_approx} is well-posed as long as the number of edge modes is sufficiently large.
In the following subsections, we present an error analysis for the two independent approximation problems \cref{eq:bubble_approx,eq:interface_approx}.

\subsection{Error estimates for the bubble approximation}
For the statement of the approximation result, let us introduce the $L^2$-projection 
 $P^j_{I_j}:L^2(\Omega)\to V_{I_j}^j$, given by
\begin{equation} \label{eq:L2_projection}
P^j_{I_j}f=\sum_{i=1}^{I_j}(\kappa^2 f,b_i^j)_{\Oj} b_i^j.   
\end{equation}
The error analysis is then straight-forward, and the proof is given for convenience of the reader.
\begin{theorem} \label{thm:bubble_convergence_rate} 
Assume $f \in L^2(\Omega)$ and set $f_\kappa:=f/\kappa^2$.
For $\uB \in \VB$ and $\uhB \in \VhB$, defined in \cref{eq:bubble,eq:bubble_approx}, respectively, we have the estimates:
\begin{align}
\|\kappa(\uB-\uhB)\|_{L^2(\Omega)}^2 
&\leq   \sum_{j=1}^J \frac{1}{|{\lambda_*^j}-1|^2} \| \kappa (f_\kappa-P^j_{I_j}f_\kappa)\|_{L^2(\Oj)}^2, \label{eq:L2_bubble_error}
\\
\|\sqrt{a}\nabla(\uB-\uhB)\|_{L^2(\Omega)}^2  
&\leq \sum_{j=1}^J \frac{\lambda_{\#}^j}{|\lambda_{\#}^j-1|^2} \|\kappa (f_\kappa-P^j_{I_j}f_\kappa)\|_{L^2(\Oj)}^2, \label{eq:H1_bubble_error}
\end{align}
where $\lambda_*^j = {\rm argmin}\{|\lambda_{i}^{j}-1|, i\geq I_j+1\}$, 
and $\lambda_{\#}^j = {\rm argmax}\{\frac{\lambda_i^j}{|\lambda_i^j-1|^2}, i\geq I_j+1\}$.
\end{theorem}
\begin{proof}
Since $\{b_i^j\}_i$ form an orthonormal basis for $L^2(\Oj)$ with the weighted inner product $(\kappa^2 u,v)_{\Oj}$, we may assume the expansion
$$
\uB = \sum_{j=1}^J \sum_{i=1}^{\infty} u_{B,i}^j b_i^j,\qquad \text{with } { u_{B,i}^j = (\kappa^2 \uB,b_i^j)_{\Oj}}.
$$
By testing \cref{eq:bubble} with $v_{B}= b_i^j$, for $i\in\NN$, and observing that $\sesH(\uB,b_i^j)=\sesI_j (\uB,b_i^j) - (\kappa^2 \uB, b_i^j)_{\Oj}$, we infer from~\cref{eq:bubbles} that
$u_{B,i}^j = F(b_i^j)/(\lambda_i^j-1)$.
Consequently, we obtain that
$$
\uhB = \sum_{j=1}^J \sum_{i=1}^{I_j} u_{B,S,i}^j b_i^j,\quad \text{with }  u_{B,S,i}^j = \myfrac{F(b_i^j)}{(\lambda_i^j-1)}.
$$
Therefore, $u_{B,S,i}^j=u_{B,i}^j$ for $1\leq i\leq I_j$ and $1\leq j\leq J$. 
Since $b_i^j=0$ on $\partial\Omega$, we further have that $F(b_i^j) = (f, b_i^j)_\Omega$.
Hence, we conclude that
\begin{align*}
\|\kappa(\uB-\uhB)\|_{L^2(\Omega)}^2 
=\sum_{j=1}^J \sum_{i=I_j+1}^{\infty} \frac{1}{|\lambda_i^j-1|^2}|(f,b_i^j)_{\Oj}|^2,
\end{align*}
which, by \cref{eq:L2_projection} and the definition of $\lambda_*^j$, implies~\cref{eq:L2_bubble_error}.
Estimate \cref{eq:H1_bubble_error} follows similarly from $\sesI_j (b_i^j,b_i^j)=\lambda_i^j$ and
\begin{align*}
\|\sqrt{a}\nabla(\uB-\uhB)\|_{L^2(\Omega)}^2 
= \sum_{j=1}^J \sum_{i=I_j+1}^\infty \frac{\lambda_i^j}{|\lambda_i^j-1|^2}|(f,b_i^j)_{\Oj}|^2.
\end{align*}
\end{proof}
\begin{remark}\label{rem:local_bubble}
%
We note that, if $f_{\mid \Oj}=0$, then $u_{ B\mid\Oj}=0$, and the approximation error vanishes on $\Oj$. Therefore, no bubble basis functions have to be computed on the corresponding domain $\Oj$. If, on the other hand, $f$ does not vanish on $\Oj$, the projection error in the estimates in \Cref{thm:bubble_convergence_rate} might be used to adaptively choose the number of required bubble functions that guarantee a certain error bound; cf. the discussion before \cite[Prop.~3.6]{Hetmaniuk2014} for the elliptic case.
\end{remark}

\begin{remark}
From the proof of \Cref{thm:bubble_convergence_rate} we see that, per subdomain, the bubble approximation is a projection of the bubble part of the solution to Helmholtz equation with respect to the norms in \cref{eq:L2_bubble_error}, \cref{eq:H1_bubble_error}. Since a projection constitutes the best-approximation and \Cref{ass:eig_not_zero} can be interpreted as a resolution condition, the bubble approximation does not suffer from the pollution effect in the sense of \cite{BabuskaSauter1997}.
\end{remark}

\subsection{Well-posedness of the interface Galerkin problem}\label{sec:well_posedness_interface}
Well-posedness of the Galerkin problem and quasi-best approximation results under an adjoint approximability condition follow as in \cite[Section 4]{GrahamSauter2019}, which has also been used in \cite[Section~4.2]{ChenHouWang2023} or \cite[Section 3.2]{MaAlberScheichl2023} in a multiscale context. To state the result let us introduce $T^*\chi=z$, mapping $\chi \in L^2(\Omega)$ to the solution $z\in  \HDO$ of the dual problem
\begin{align}\label{eq:adjoint}
\sesH(v,z) = (v,\chi)\quad\text{for all } v\in \HDO,
\end{align}
which is well-defined due to \cref{thm:well-posedness_Helmholtz}.
Moreover, we denote with $T_\Gamma^*\chi=\ExtG(z_{\mid\Gamma})\in\VG$ the interface component of $T^*\chi$. Using \cref{eq:ortho}, we observe that $T_\Gamma^*\chi$ is a solution to \cref{eq:adjoint} for test functions $v\in\VG$.
The approximation properties of the interface space $\VhG$ are measured by the adjoint approximability constant
\begin{align}\label{eq:adjoint_approx_constant}
\sigma^* = \sup_{\varphi\in L^2(\Omega)\setminus\{0\}} \frac{\inf_{\vG\in \VhG} \|T_\Gamma^*( \kappa^2\varphi)-\vG\|_\sesP}{\|\kappa\varphi\|}.
\end{align}
With these preparations, we can state the abstract result \cite[Thm.~4.2]{GrahamSauter2019} in our setting.
\begin{lemma}\label{lem:well-posedness_interface_approx}
Suppose that $C_{\sesH}\sigma^*\leq \myfrac{1}{2}$, where $C_{\sesH}$ is defined in \cref{eq:boundedness_sesquilinearH} and $\sigma^*$ is the adjoint approximability constant. Then, the Galerkin problem~\cref{eq:interface_approx} has a unique solution, and the following estimates hold:
\begin{align}
\|u_\Gamma-\uhG\|_\sesP &\leq 2 C_{\sesH} \inf_{\vG\in \VhG}\|u_\Gamma - \vG\|_\sesP, \label{eq:sesP_estimate_interface}\\
\|\kappa (u_\Gamma-\uhG)\|&\leq 2 C_{\sesH}^2 \sigma^* \inf_{\vG\in \VhG}\|u_\Gamma - \vG\|_\sesP.\label{eq:L2_estimate_interface}
\end{align}
\end{lemma}
Note that the quasi-optimality constant $2C_{\sesH}$ in \cref{eq:sesP_estimate_interface} is independent of $\omega$ for large $\omega$; see \cref{thm:well-posedness_Helmholtz}. However, the result applies only if $\sigma^*$ is sufficiently small, which requires sufficiently many edge modes, see \Cref{sec:4.4}.
\begin{proof} The proof follows closely \cite{GrahamSauter2019} and is included for convenience of the reader.
Denote $e_\Gamma=u_\Gamma-\uhG$, and let $\psi=T_\Gamma^*(\kappa^2 e_\Gamma)$.
Then, for $\psi_{\Gamma,S}\in \VhG$ being the best-approximation of $\psi$ in the $\sesP$-norm, Galerkin orthogonality yields
\begin{align*}
\|\kappa e_\Gamma\|^2 = \sesH(e_\Gamma,\psi) = \sesH(e_\Gamma,\psi-\psi_{\Gamma,S})\leq C_{\sesH} \sigma^* \|e_\Gamma\|_\sesP \|\kappa e_\Gamma\|,
\end{align*}
where we used the adjoint approximability \cref{eq:adjoint_approx_constant} as follows:
\begin{align*}
\|\psi-\psi_{\Gamma,S}\|_\sesP 
= \inf_{\vG\in \VhG}\|T_\Gamma^*(\kappa^2e_\Gamma)-\vG\|_\sesP
\leq\sigma^* \|\kappa e_\Gamma\|. 
\end{align*}
Using Galerkin orthogonality once again, we obtain for arbitrary $\vG\in \VhG$ that
\begin{align*}
\|e_\Gamma\|_\sesP^2 &= \Re\{\sesH(e_\Gamma,e_\Gamma)\}+ 2 \|\kappa^2 e_\Gamma\|^2
= \Re\{\sesH(e_\Gamma,u_\Gamma-\vG)\}+ 2 \|\kappa^2 e_\Gamma\|^2\\
&\leq C_{\sesH} \|e_\Gamma\|_\sesP \|u_\Gamma - \vG\|_\sesP + 2 (C_{\sesH}\sigma^*)^2 \| e_\Gamma\|_\sesP^2.
\end{align*}
 The abstract error estimates \cref{eq:sesP_estimate_interface} and \cref{eq:L2_estimate_interface} then follow using the assumption $\sigma^* \leq \myfrac{1}{(2 C_{\sesH})}$.
Similarly, we can prove the existence of the Galerkin projection. Assume
$$
\sesH(\uhG,\vG)=0\quad \text{for all } \vG\in \VhG.
$$
We have to show that $\uhG=0$. As before, but using the previous identity and choosing $\vG$ as the best-approximation in the $\sesP$-norm of $T_\Gamma^*(\kappa^2 \uhG)$, we obtain
\begin{align*}
\|\kappa \uhG\|^2 =
\sesH(\uhG,T_\Gamma^*(\kappa^2 \uhG)-\vG)
\leq C_{\sesH} \sigma^* \|\uhG\|_\sesP \|\kappa \uhG\|.
\end{align*}
Therefore, we have that
\begin{align*}
\|\uhG\|_\sesP^2 &= \Re\{\sesH(\uhG,\uhG)\}+ 2 \|\kappa^2 \uhG\|^2 = 2 \|\kappa^2 \uhG\|^2
\leq 2 (C_{\sesH}\sigma^*)^2 \|\uhG\|_\sesP^2,
\end{align*}
which, by assumption $C_{\sesH}\sigma^*\leq \myfrac{1}{2}$, yields that $\uhG=0$. 
\end{proof}

\subsection{Estimates of the best-approximation error in \texorpdfstring{$\VhG$}{VGS}}\label{sec:Interface_approximation_error}
Quantitative estimates for the interface approximation error follow from localizing the error to single edges and applying the estimates proven in \cref{lem:approximation_properties_edge}.
\begin{theorem}\label{thm:estimate_sesP_norm}
Assume that the solution $u$ to \cref{eq:HelmholtzVariational} satisfies $u \in H^2(e)$ for all $e\in\E$,
and denote by $\uhG$ the solution to \cref{eq:interface_approx}. Then, 
\begin{align*}
\|u_\Gamma-\uhG\|_\sesP \leq \sum_{e\in\E} \frac{C}{(\lambda_{I_e+1}^e)^{\myfrac{3}{4}}}\|\Delta_e u-P^e_{I_e}\Delta_e u\|_{L^2(e)}
\end{align*}
for a constant $C>0$, bounded by $O(\|\kappa\|_\infty)$, independent of $u$ and $\uhG$. 
\end{theorem}
\begin{proof}
According to \cref{lem:well-posedness_interface_approx}, it suffices to estimate the best-approximation error $\inf_{\vG}\|u_\Gamma-\vG\|_\sesP$.
By continuity of $u$, the nodal interpolant $I_\V u$ of $u$ is well-defined and, by construction, we obtain that $(u-I_\V u)_{\mid e} \in \HOe$.
These observations motivate the choice
\begin{align}\label{eq:interpolant}
\vG = \ExtG I_\V u + \sum_{e\in\E} \ExtE P^e_{I_e}(u-I_\V u),
\end{align}
with $P^e_{I_e}$ introduced in \cref{eq:l2_projection_edge}.
We then have that $u_\Gamma(p)-\vG(p)=0$ for all $p\in \V$. 
Moreover, in view of \cref{sec:function_spaces},
$(\uG-\vG)_{\mid e}\in  \HOe\subset \HtrELM$.
Therefore, the error can be localized to single edges as follows
\begin{align}\label{eq:localization}
{\uG}_{\mid\Oj}  - {\vG}_{\mid\Oj} = \ExtLoc{j}( u_{\mid\partial\Oj} - {\vG}_{\mid\partial\Oj} ) = \sum_{e\in\E, e\subset\partial\Oj} (\ExtE (u_{\mid e} - {\vG}_{\mid e}))_{\mid\Oj}.
\end{align}
From the latter identity and from the definition \cref{eq:def_sesP}, we obtain the following estimate
\begin{align}
\|\uG - \vG\|_\sesP^2 
&\leq C \sum_j \sum_{e\in\E,e\subset\partial\Oj} \|\ExtE (u_{\mid e}- {\vG}_{\mid e})_{\mid\Oj}\|_{\sesP}^2 \notag\\
&\leq (a_{\max}+\|\kappa\|_\infty^2)C \max_{j}(1+\myfrac{1}{\beta^j})^2\sum_{e\in\E} \|u-\vG\|_{\HtrELM}^2, \label{eq:localization2}
\end{align}
where $C$ depends on the number of edges of a subdomain $\Oj$ and where we used \cref{eq:extension_bound} and equivalence of the $\sesP$-norm to the $\Hj$-norm.
The interpolation estimate \cref{eq:interpolation,lem:approximation_properties_edge} then yield the assertion.
\end{proof}
%
%
If the solution has more regularity, the convergence improves, as long as the derivatives of the edge modes do not grow too quickly, which is true, e.g., if $e$ is a line segment.
\begin{lemma}\label{lem:improved_rate_interface}
 In addition to the assumptions of \cref{thm:estimate_sesP_norm}, suppose that $u\in H^3(e)$ for all $e\in\E$. Moreover, suppose that there is a constant $C>0$ such that $|\partial_e\tau_i^e(p)|\leq C \sqrt{\lambda_i^e}$ for all $i\in\NN$ and $p\in\partial e$.
 Then there is a constant $C>0$ such that
\begin{align*}
\|u_\Gamma-\uhG\|_\sesP \leq C \sum_{e\in\E} \frac{1}{\lambda_{I_e+1}^e}\| u\|_{H^3(e)}.
\end{align*}
\end{lemma}
\begin{proof}
In view of \cref{thm:estimate_sesP_norm} it suffices to bound 
\begin{align*}
\|\Delta_e u-P^e_{I_e}\Delta_e u\|_{L^2(e)} = \left(\sum_{i>I_e}|(\Delta_e u,\tau_i^e)_e|^2\right)^{1/2}.
\end{align*}
Since $-\Delta_e \tau_i^e=\lambda_i^e\tau_i^e$ on $e$ by elliptic regularity, integration by parts yields that
\begin{align*}
(\Delta_e u,\tau_i^e)_e =- \frac{1}{\lambda_i^e} (\Delta_e u,\Delta_e\tau_i^e)_e
= \frac{1}{\lambda_i^e} \left((\partial_e^3 u,\partial_e\tau_i^e)_e+ \left[\Delta_e u\partial_e\tau_i^e \right]_{q}^p\right),
\end{align*}
where $p,q$ denote the endpoints of $e$.
Therefore, using the Cauchy-Schwarz inequality, the continuity of the embedding $\He\hookrightarrow C^0(\overline{e})$ \cite[p.~97]{Adams75}, we have that
\begin{align*}
|(\Delta_e u,\tau_i^e)_e| \leq \frac{C}{\sqrt{\lambda_i^e}} \|u\|_{H^3(e)}.
\end{align*}
Then, by using \cref{eq:edge_eigs_behavior}, we have that
\begin{align*}
 \sum_{i>I_e}|(\Delta_e u,\tau_i^e)_e|^2 \leq C\|u\|_{H^3(e)}^2\sum_{i>I_e} \frac{1}{\lambda_i^e} \leq C \frac{\|u\|_{H^3(e)}^2}{\sqrt{\lambda^e_{I_e+1}}},
\end{align*}
which proves the claim.
\end{proof}

\begin{remark}\label{rem:regularity}
    The error estimates in \Cref{thm:estimate_sesP_norm} and \Cref{lem:improved_rate_interface} assume regularity of the solution to the Helmholtz equation along the edges of the domain decomposition. By the trace theorem \cite{Grisvard2011}, this smoothness follows from regularity of the solution in a neighborhood of the edges, which can be established under suitable assumptions on the coefficients; see, e.g., \cite[Chapter~8]{GT2001} or \cite[Chapter~4]{Grisvard2011} and also \Cref{sec:4.4}. 
    If the solution $u$ enjoys less regularity than required in the above statements, corresponding error estimates can still be proven.
    For instance, under the regularity $u\in W^{1,p}(\Omega)$, for $p>2$, established above, the corresponding estimate to \Cref{thm:estimate_sesP_norm} will have the exponent $(p-2)/(4p)>0$ instead of $3/4$.
    To derive the latter statement, we use that $H^{1-1/p}_{0}(e)$ is an interpolation space between $L^2(e)$ and $H^1_0(e)$ \cite[Ch.~1, Thm.~11.6]{LionsMagenes1972}, $H^{1/2}_{00}(e)$ is an interpolation space between $L^2(e)$ and $H^{1-1/p}_0(e)$ \cite[Ch.~1, Thm.~11.7]{LionsMagenes1972}, and that $(u-I_\V u)_{\mid e}\in W^{1-1/p,p}_0(e)\subset H^{1-1/p}_{0}(e)\subset H^{1/2}_{00}(e)$ can be approximated in $V^e$.
\end{remark}

\subsection{Estimates for adjoint approximability constant}\label{sec:4.4}
The well-posedness and the error estimates for the interface problem in \Cref{sec:well_posedness_interface} relied on the smallness of the adjoint approximability constant $\sigma^*$ defined in \cref{eq:adjoint_approx_constant}.
In order to estimate $\sigma^*$ in our setting, we follow the ideas of~\cite{GrahamSauter2019}, 
where they work with piecewise linear finite elements, requiring $H^2(\Omega)$-regularity of the adjoint problem \cref{eq:adjoint}. 
Since we need to consider the interface problem only, we can require a weaker $H^2$-regularity in the vicinity of the interface $\Gamma$. Let, for some fixed $\delta>0$, 
\begin{align}\label{eq:Omega_G}
\Omega_{\Gamma}\subset\{x\in\Gamma: {\rm dist}(x,\Gamma\cup \partial\Omega)<\delta\} 
\end{align}
denote an open neighborhood of $\Gamma\cup\partial\Omega$ such that $\partial\Omega_\Gamma\setminus\partial\Omega$ is smooth,  see \Cref{fig:sketch}. 
We start with a result similar to \cref{thm:estimate_sesP_norm} but with slightly weaker regularity assumptions.

\begin{lemma}\label{lem:best_approx_H2}
Let $\lambda^\Gamma=\min_{e}\lambda_{I_e}^e$. Then, there exists a constant $C>0$ with $C=O(\|\kappa\|_\infty)$ such that, for all $z\in H^2(\Omega_\Gamma)$, it holds that
\begin{align}\label{eq:estimate_adjoint_approximability}
\inf_{\vG\in \VhG}\|\ExtG(z_{\mid\Gamma})-\vG\|_{\sesP}
\leq \frac{C}{\sqrt{\lambda^\Gamma}}\|z\|_{H^{2}(\Omega_\Gamma)}.
\end{align}
\end{lemma}
\begin{proof}
In view of \cref{eq:localization2}, which also introduces the $\kappa$ dependency of the constants, it suffices to estimate
\begin{align*}
\inf_{v^e\in V_{I_e}^e}\| z -I_\V z - v^e \|_{\HtrELM}
\end{align*}
in terms of $\|z\|_{H^2(\Omega_\Gamma)}$.
By embedding, $z\in H^2(\Omega_\Gamma)$ implies that $z\in H^{\myfrac{3}{2}}(e)$ for all $e\in\E$, and applying \cref{lem:approximation_properties_edge2} yields that
\begin{align*}
\inf_{\vG\in \VhG}\|z-\vG\|_{\HtrELM}
\leq \frac{C}{\sqrt{\lambda_{I_e+1}^e}}\|z\|_{H^{\myfrac{3}{2}}(e)}.
\end{align*}
Hence, the definition of $\lambda^\Gamma$ yields the assertion.
\end{proof}
Assuming regularity of the coefficients of the Helmholtz problem locally around the interface $\Gamma$, we can next estimate the adjoint approximability constant.


\begin{theorem}\label{thm:estimate_sigma}
If $a\in C^{0,1}(\overline{\Omega_\Gamma})$ and $\beta\in C^{0,1}(\overline{\Gamma_R})$,
then there exists a constant $C>0$ such that $C=O(\|\kappa\|_\infty^2 C_{\rm stab})$, with $C_{\rm stab}$ from \cref{eq:Cstab}, and
\begin{align}\label{eq:est_adj_approx_2}
    \sigma^* \leq \myfrac{C}{\sqrt{\lambda^\Gamma}}.
\end{align}
\end{theorem}
\begin{proof}
In view of \cref{eq:adjoint_approx_constant}, the key for estimating $\sigma^*$
is to obtain an estimate for
$$
\inf_{\vG\in \VhG}\|T_\Gamma^*(\kappa^2\varphi)-\vG\|_\sesP
$$
in terms of $\|\kappa\varphi\|$ for arbitrary $\varphi\in L^2(\Omega)$, where $T_\Gamma^*(\kappa^2\varphi)\in \VG$ is the interface component of the solution to the adjoint problem \cref{eq:adjoint}.
Using \cref{eq:estimate_adjoint_approximability}, we have the bound
$$
\inf_{\vG\in \VhG}\|T_\Gamma^*(\kappa^2\varphi)-\vG\|_\sesP\leq \frac{C}{\sqrt{\lambda^\Gamma}}\| T^*(\kappa^2 \varphi)\|_{H^2(\Omega_\Gamma)},
$$
with $C=O(\|\kappa\|_\infty)$, provided $T^*(\kappa^2 \varphi)\in H^2(\Omega_\Gamma)$.
Next, we show that there exists a constant $C_*>0$ such that $C_*=O(\|\kappa\|_\infty C_{\rm stab})$, with $C_{\rm stab}$ from \cref{eq:Cstab}, and
\begin{align}\label{eq:est_adjoint}
 \| T^*(\kappa^2 \varphi)\|_{H^2(\Omega_\Gamma)}\leq C_* \|\kappa \varphi\|_{L^2(\Omega)},
\end{align}
which will conclude the proof.
Let $z=T^*(\kappa^2\varphi)$ denote the solution of the dual problem with $G(v)=(v,\kappa^2\varphi)$, $v\in \HDO$.
Similar to the comments at the end of \Cref{sec:well-posedness-Helmholtz}, we deduce that $z\in C^0(\overline{\Omega})$.
For any $\Omega'\Subset \Omega$, \cite[Theorem~8.8]{GT2001} ensures that 
\begin{align}\label{eq:H2_bound1}
\|z\|_{H^2(\Omega_\Gamma\cap\Omega')}\leq C\|\kappa\|_\infty( \|z\|_{\sesP} + \|\kappa\varphi\|_{L^2(\Omega)}),
\end{align}
with a constant $C$ depending on $a_{\text min}$, $\|a\|_{C^{0,1}}$, and the distance of $\Omega_\Gamma\cap\Omega'$ to $\partial\Omega$ but independent of $\kappa$.
Rewriting the dual problem into a Poisson problem with mixed Dirichlet-Neumann boundary conditions
and applying \cite[Corollary~4.4.3.8]{Grisvard2011}, we deduce $H^2$-regularity of $z$ also in a neighborhood of $\Gamma_R\cap\overline{\Omega_\Gamma}$ as follows. Consider a connected component $\Gamma_0$ of $\Gamma_R$. Let $\Omega_0\subset\Omega_\Gamma$ be a polygonal domain such that $\Gamma_0\subset\partial\Omega_0$ and that $\partial\Omega_0\setminus\partial\Omega$ is smooth and, if $\Gamma_D\neq \emptyset$, the angles at the boundaries of $\partial\Omega_0\setminus\partial\Omega$ are strictly convex, see \Cref{fig:sketch}. 
\begin{figure}[ht!]
    \centering
    \begin{tikzpicture}[scale=0.75]
	\begin{scope}[]
    \path [clip,draw] (0,3) -- (3,0) -- (8,0) -- (11,3) -- (11,6) -- (8,9) -- (3,9) -- (0,6) -- (0,3);
	\draw [line width=1cm,black!10] (0,3) -- (11,3);
	\draw [line width=1cm,black!10] (0,6) -- (11,6);
	\draw [line width=1cm,black!10] (3,0) -- (3,9);
	\draw [line width=1cm,black!10] (8,0) -- (8,9);

	\draw [line width=1cm,black!10]  (0,3) -- (3,0) -- (8,0) -- (11,3);
	\draw [line width=0.5cm,black!30,cap=round]  (0,3) -- (3,0) -- (8,0) -- (11,3); 
	
    \draw (0,3) -- (3,0) -- (8,0) -- (11,3) -- (11,6) -- (8,9) -- (3,9) -- (0,6) -- (0,3);
	\draw [dashed] (0,3) -- (11,3);
	\draw [dashed] (0,6) -- (11,6);
	\draw [dashed] (3,0) -- (3,9);
	\draw [dashed] (8,0) -- (8,9);
	
	\draw [line width=0.1cm,cap=round]  (0,3) -- (3,0) -- (8,0) -- (11,3); 
	\draw (3.35,3.3) node{$\Omega_\Gamma$};
	\draw (3.35,0.5) node{$\Omega_0$};
	\end{scope}
	
	\draw (9.5,0.75) node{$\Gamma_R$};
	\draw (9.5,8.5) node{$\Gamma_D$};
  
    \end{tikzpicture}
    \caption{A sketch of a polygonal domain $\Omega$ occupying an octagon partitioned into $9$ subdomains $\Omega_j$, $j=1,\ldots,9$. The portion $\Gamma_R$ is depicted in bold. The dashed lines and the edges in $\Gamma_R$ indicate the set of edges in $\mathcal{E}$. The light grey area corresponds to $\Omega_\Gamma$ defined in \cref{eq:Omega_G}. The dark grey area corresponds to the domain $\Omega_0$ used in the proof of \Cref{thm:estimate_sigma}.\label{fig:sketch}}
\end{figure}
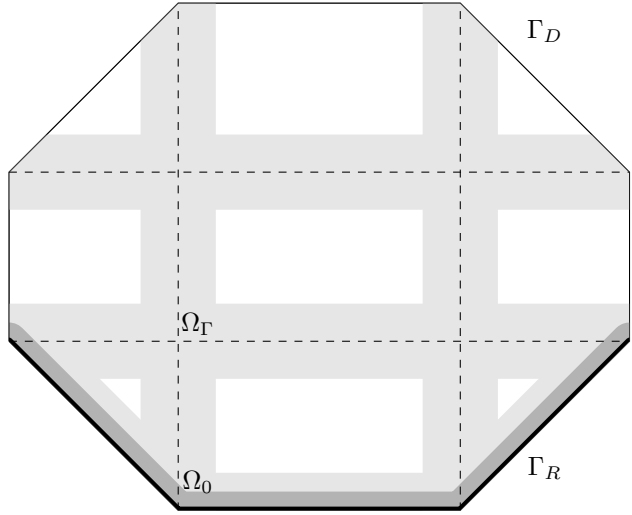
Note that, if $\Gamma_D\neq \emptyset$, then $z$ is $H^2$-regular in a neighborhood of $\partial\Omega_0\setminus\partial\Omega$ satisfying a bound corresponding to~\cref{eq:H2_bound1}; cf.~\cite[Theorem~8.12]{GT2001}. We observe that $z$ is the weak solution of
\begin{align*}
    -\Delta z + z &= a^{-1}\left(\kappa^2\varphi + (\kappa^2+a) z + \nabla a\cdot\nabla z\right) =:\tilde f \quad\text{in } \Omega_0,\\
    \partial_n z &= -ia^{-1}\omega\beta z\quad\text{on } \Gamma_0,\\
    z&=0 \quad\text{on } \Gamma_D\cap\partial\Omega_0,\\
    z&=z \quad\text{on } \Omega\cap \partial\Omega_0,
\end{align*}
with $\tilde f\in L^2(\Omega)$.
Using \Cref{ass:parameters}(i) and continuity of $z$ on $\partial\Omega_0$, we deduce that $z\in H^2(\Omega_0)$ \cite[Corollary~4.4.3.8]{Grisvard2011}.
In particular, $\partial_n z$ is in $H^{1/2}$ on each segment of $\partial\Omega_0$ \cite[Theorem~1.5.2.8]{Grisvard2011}.
Using a lifting of $\partial_n z$ as in \cite[Lemma~2.12]{Gander2015},  we then obtain the bound
\begin{align}\label{eq:H2_bound2}
\|z\|_{H^2(\Omega_0)}\leq C \left((1+\|\kappa\|_\infty) \big( \|z\|_{\sesP} + \|\kappa\varphi\|_{L^2(\Omega)} \big)+ \|\partial_n z\|_{H^{1/2}_{pw}(\partial\Omega_0)}\right),
\end{align}
where  $\|\partial_n z\|_{H^{1/2}_{pw}(\partial\Omega_0)}$ denotes the sum of $H^{1/2}$-norms over the corresponding segments of $\partial\Omega_0$.
The $H^{1/2}(\Omega\cap\Omega_0)$-norm of $\partial_n z$ can be bounded by the right-hand side of \cref{eq:H2_bound1} using the previous considerations. 
The $H^{1/2}$-norms on the segments of $\Gamma_0$ are bounded, using the boundary condition $\partial_n z = -ia^{-1}\omega\beta z$, by the corresponding $H^{1/2}$-norms of $z$, which is again bounded by the terms on the right-hand side of \cref{eq:H2_bound1}.
To bound $\|\partial_n z\|_{H^{1/2}(\Gamma_D\cap\partial\Omega_0)}$, consider  $\Omega''\Subset\Omega_0\cup(\Gamma_D\cap\partial\Omega_0)$. An application of \cite[Theorem~9.13]{GT2001} yields a constant $C>0$ independent of $\kappa$ such that
\begin{align}\label{eq:H2_bound4}
\|z\|_{H^2(\Omega'')}\leq C\|\kappa\|_\infty \left(\|z\|_{\sesP}+\|\kappa\varphi\| \right).
\end{align}
An application of the trace theorem \cite[Theorem~1.5.2.8]{Grisvard2011}, yields a bound for $\|\partial_n z\|_{H^{1/2}(\Gamma_D\cap\partial\Omega_0)}$ in terms of the right-hand side in \cref{eq:H2_bound4}.

Combining the above estimates with \cref{eq:Cstab}, we infer that
\begin{align}\label{eq:H2_bound3}
\|z\|_{H^2(\Omega_\Gamma)}\leq C\|\kappa\|_\infty C_{\rm stab} \|\kappa\varphi\|_{L^2(\Omega)},
\end{align}
with a constant $C$ independent of $\omega$. Hence \cref{eq:est_adjoint} holds with $C_*=C\|\kappa\|_\infty C_{\rm stab}$.
\end{proof}
Since the method and its analysis presented here employ a fixed domain decomposition, the regularity assumption on the coefficients required in \Cref{thm:estimate_sigma} might be verified in certain applications; see, e.g., the periodic structure in \Cref{sec:periodic} below.
A consequence of the previous result is that, by using sufficiently many edge modes, the assumptions of \cref{lem:well-posedness_interface_approx} can be verified. More precisely, if $C_{\rm stab}$ is independent of $\omega$, which holds in certain cases \cite[Theorem~4.5]{GrahamSauter2019}, we can use \cref{eq:edge_eigs_behavior} and $\sigma^*=O(C_{\rm stab}\|\kappa\|^2\max_e|e|/\min_e I_e)$ to infer that the number of modes should scale like $I_e \geq C \|\kappa\|^2_\infty |e|$ for stability.
We obtain the following statement.

\begin{theorem}\label{thm:convergence_interface_L2}
In addition to the assumptions of \cref{thm:estimate_sesP_norm,thm:estimate_sigma} suppose that $u\in H^{2+\alpha}(e)$ for all $e\in\E$, for $\alpha=0,1$.
Then, for some $C>0$, and $\lambda^\Gamma$ sufficiently large, it holds
\begin{align*}
\|\kappa (\uG-\uhG)\| \leq \frac{C}{(\lambda^\Gamma)^{\myfrac{5}{4} + \myfrac{\alpha}{4}}} \sum_{e\in\E} \|u\|_{H^{2+\alpha}(e)}.
\end{align*}
\end{theorem}
\begin{proof}
Inserting the estimate \cref{eq:est_adj_approx_2} for $\sigma^*$ and the error bounds stated in \cref{thm:estimate_sesP_norm} (if $\alpha=0$) or \cref{lem:improved_rate_interface} (if $\alpha=1$) into \cref{eq:L2_estimate_interface} yields the assertion. 
\end{proof}

\section{Numerical results} \label{sec:results}
We provide numerical experiments to support our theoretical results and show the effectiveness of our approach.
Since, in general, we cannot compute the ACMS basis functions in~\cref{eq:bubble_space_finite,eq:interface_space_finite} analytically, we compute them approximately using an underlying finite element discretization that adopts piecewise linear and continuous functions. To do so, we employ a quasi-uniform triangulation of the computational domain $\Omega$ into (non-curved) triangles, such that each subdomain $\Oj$ is the union of elements of that triangulation; see, for instance,~\cref{fig:circular_mesh}. 
We denote the corresponding finite element solution of the Helmholtz problem by $u_{\rm FEM}$. Moreover, the resulting errors are quantified according to~\cref{tab:errors}.
\begin{table}[t] \centering 
\caption{Different errors in the $L^2$- (left) and $H^1$-norms (right) discussed in~\cref{sec:results} row by row: absolute approximation errors of the ACMS solution $\uh$ with respect to the exact solution $u$, absolute approximation errors of the ACMS solution $\uh$ with respect to the finite element solution $u_{\rm FEM}$, relative approximation errors of the ACMS solution $\uh$ with respect to the finite element solution $u_{\rm FEM}$, approximation errors of the finite element solution $u_{\rm FEM}$ with respect to the exact solution $u$, and interpolation error of the nodal interpolant $I_h(u)$ with respect to the exact solution $u$.
	\label{tab:errors}
	}
	\begin{tabular}{l l c l l}
		\toprule
		\multicolumn{2}{c}{$L^2$-errors} && \multicolumn{2}{c}{$H^1$-errors} \\
		\cline{1-2} \cline{4-5}
		$e_0$           & $\|u-\uh\|_{L^2(\Omega)}$                         && $e_1$           & $\|u-\uh\|_{\HO}$                         \\
		$e_{0,h}$       & $\|u_{\rm FEM}-\uh\|_{L^2(\Omega)}$               && $e_{1,h}$       & $\|u_{\rm FEM}-\uh\|_{\HO}$               \\
		$e_{0,h}^r$     & $\myfrac{e_{0,h}}{\|u_{\rm FEM}\|_{L^2(\Omega)}}$ && $e_{1,h}^r$     & $\myfrac{e_{1,h}}{\|u_{\rm FEM}\|_{\HO}}$ \\
		$e_{0,\rm FEM}$ & $\|u-u_{\rm FEM}\|_{L^2(\Omega)}$                 && $e_{1,\rm FEM}$ & $\|u-u_{\rm FEM}\|_{\HO}$                 \\
		$e_{0,\rm int}$ & $\|u-I_h(u)\|_{L^2(\Omega)}$                      && $e_{1,\rm int}$ & $\|u-I_h(u)\|_{\HO}$ \\
		\bottomrule
	\end{tabular} 

\end{table}

\subsection{Classical Helmholtz example}\label{sec:classical_Helmholtz_example}
Let $\Omega=B_1(0)$ be the unit disc and $\Gamma_R=\partial\Omega$.
We first consider the boundary value problem given by~\cref{eq:model_strong,eq:model_strong_bcR} with $a(x) = 1$, $\beta(x)=1$.
The source terms $f$ and $g$ are defined such that 
the problem has the plane wave solution 
	$u(x) = \exp({-\iota  k\cdot x})$,
with wave vector $k=\kappa(0.6,0.8)$ and variable wavenumber $\kappa$. 
In particular, $f(x)=0$, therefore the solution satisfies $u_B=0$, and $u$ can be approximated by solving the interface problem \cref{eq:interface_approx} only; see  \cref{thm:bubble_convergence_rate} and also~\cref{rem:local_bubble}.

We employ a decomposition of $\Omega$ as depicted in \cref{fig:circular_mesh} (left) with an initial coarse triangulation as shown in \cref{fig:circular_mesh} (right), which corresponds to a domain decompositions into $J = 8$ subdomains. In this domain decomposition, the number of edges in $\E$ is $12$ and the number of vertices in $\V$ equals $5$.
The ACMS solution $\uh$ and the finite element solution $u_{\rm FEM}$ are computed using uniform refinements of the coarse triangulation; see~\cref{fig:circular_mesh} (right). 

\begin{figure}[t]
\centering 
\includegraphics[width=0.45\textwidth]{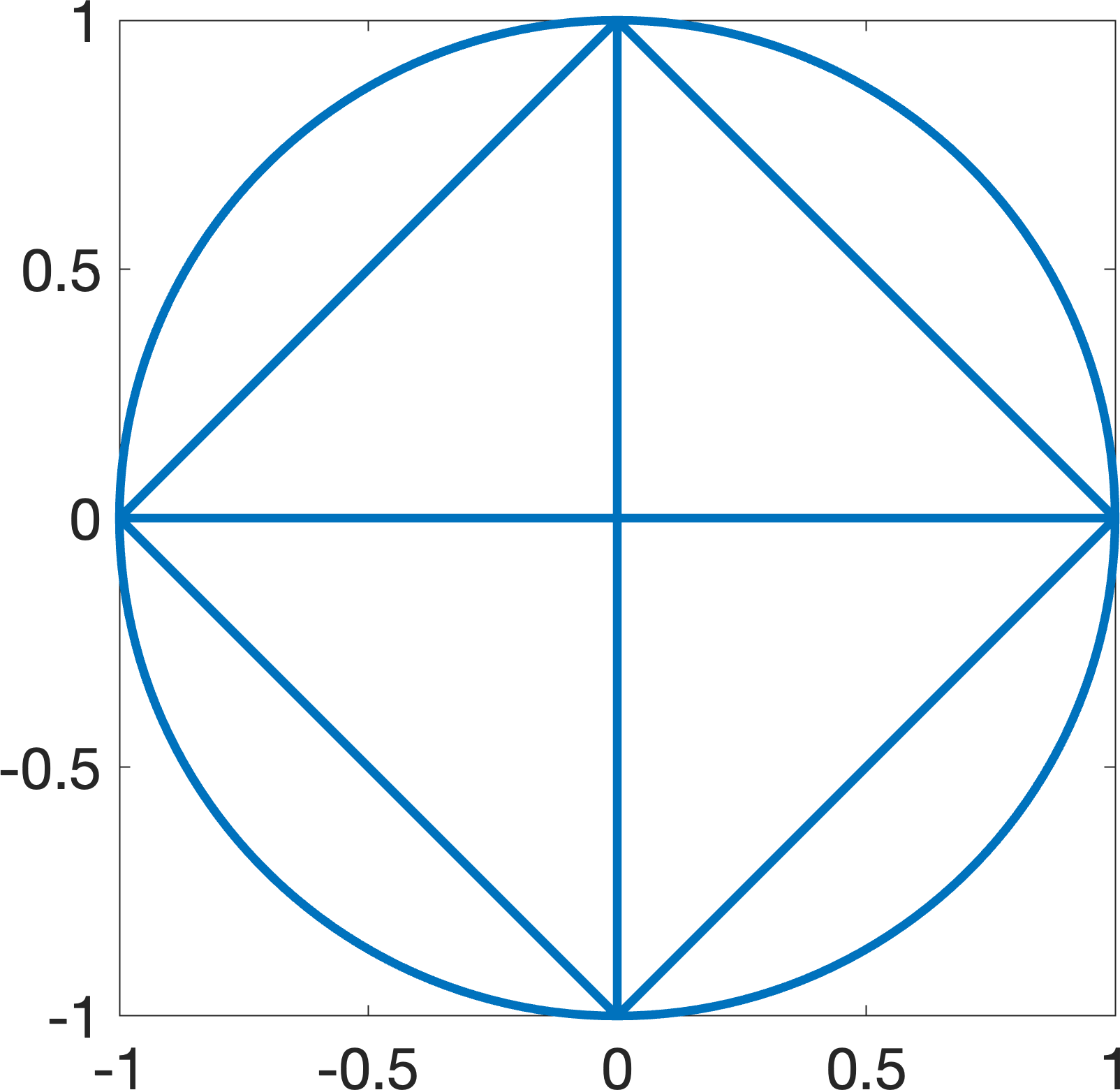}
\hfill
\includegraphics[width=0.45\textwidth]{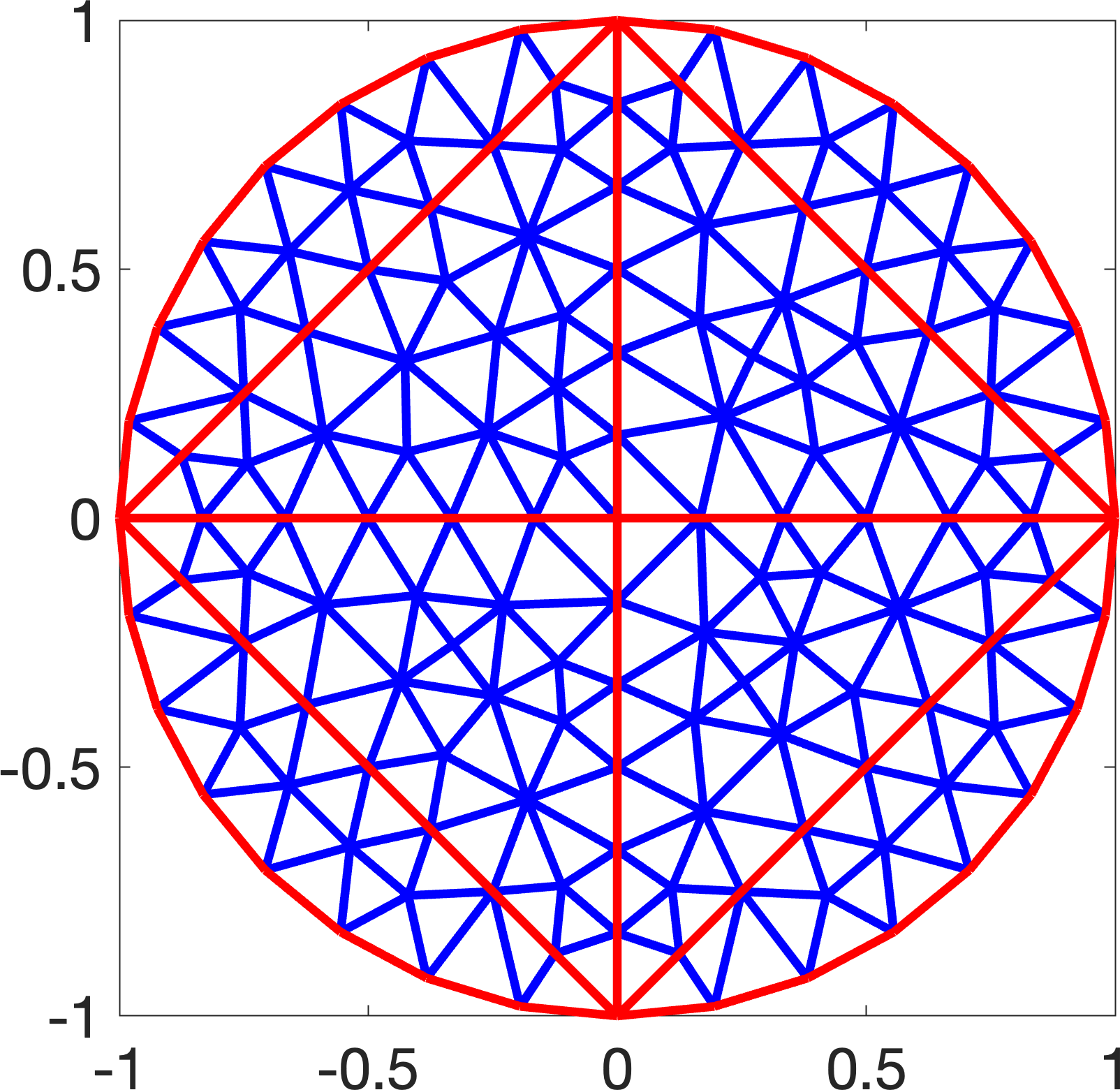}
\label{fig:circular_mesh}
\caption{\textbf{Left:} $\Omega$ is the unit disc, decomposed into $8$ subdomains. \textbf{Right:} the corresponding coarse (red) and fine (blue) finite element meshes.} 
\end{figure}

\subsubsection{Low wavenumber}
Let us first consider the case of a low wavenumber $\kappa=1$. In this setting, we obtain $\|u\|_{\HO}\approx 2.5$ and $\|u\|_{L^2(\Omega)}\approx 1.8$ for the plane wave solution $u$.
We compare the plane wave solution $u$ with the approximation $\uh$ given by the ACMS method 
for different underlying finite element discretizations and multi-indices $S_\Gamma$, which define the order of the interface approximation space $V_{\Gamma,S_\Gamma}$ introduced in \cref{eq:interface_space_finite}.
Moreover, we use the same value of  $I_e \in \{2,4,8,\ldots, 64\}$ for each of the $12$ edges in $\E$.
 
\begin{table}[t] 
	\centering\small\setlength\tabcolsep{0.55em}
	\caption{Classical Helmholtz example with $\kappa=1$: errors $e_{0}$ (top) and $e_{1}$ (bottom) as defined in~\cref{tab:errors}
		for different number of edge modes $|S_\Gamma|\in\{24,28,\ldots,768\}$, computed on different finite element meshes with mesh size $h$, and the corresponding interpolation error for the nodal interpolant. \label{tab:convergence_edge_modes_classical}}
	\begin{tabular}{l l l l l l l l}
		\toprule
        && \multicolumn{6}{c}{$|S_\Gamma|$}\\
        \cline{3-8}
		\multicolumn{1}{c}{$h$} & \multicolumn{1}{c}{$e_{0,\rm int}$} & \multicolumn{1}{c}{24} & \multicolumn{1}{c}{48} & \multicolumn{1}{c}{96} & \multicolumn{1}{c}{192} & \multicolumn{1}{c}{384} & \multicolumn{1}{c}{768} \\ \hline
		$1.3{\cdot} 10^{-3}$    & $7.6{\cdot} 10^{-7}$                & $4.1{\cdot} 10^{-3}$   & $7.5{\cdot} 10^{-4}$   & $1.1{\cdot} 10^{-4}$   & $1.5{\cdot} 10^{-5}$    & $2.5{\cdot} 10^{-6}$    & $1.5{\cdot} 10^{-6}$    \\
		$6.9{\cdot} 10^{-4}$    & $1.9{\cdot} 10^{-7}$                & $4.1{\cdot} 10^{-3}$   & $7.5{\cdot} 10^{-4}$   & $1.1{\cdot} 10^{-4}$   & $1.5{\cdot} 10^{-5}$    & $2.0{\cdot} 10^{-6}$    & $4.6{\cdot} 10^{-7}$    \\
		$3.4{\cdot} 10^{-4}$    & $4.7{\cdot} 10^{-8}$                & $4.1{\cdot} 10^{-3}$   & $7.5{\cdot} 10^{-4}$   & $1.1{\cdot} 10^{-4}$   & $1.5{\cdot} 10^{-5}$    & $1.9{\cdot} 10^{-6}$    & $2.6{\cdot} 10^{-7}$    \\ \bottomrule\toprule
		\multicolumn{1}{c}{$h$} & \multicolumn{1}{c}{$e_{1,\rm int}$} & \multicolumn{1}{c}{24} & \multicolumn{1}{c}{48} & \multicolumn{1}{c}{96} & \multicolumn{1}{c}{192} & \multicolumn{1}{c}{384} & \multicolumn{1}{c}{768} \\ \hline
		$1.3{\cdot} 10^{-3}$    & $1.0{\cdot} 10^{-3}$                & $5.7{\cdot} 10^{-2}$   & $1.6{\cdot} 10^{-2}$   & $4.3{\cdot} 10^{-3}$   & $1.5{\cdot} 10^{-3}$    & $1.1{\cdot} 10^{-3}$     & $1.0{\cdot} 10^{-3}$     \\
		$6.9{\cdot} 10^{-4}$    & $5.4{\cdot} 10^{-4}$                & $5.7{\cdot} 10^{-2}$   & $1.6{\cdot} 10^{-2}$   & $4.2{\cdot} 10^{-3}$   & $1.2{\cdot} 10^{-3}$    & $6.0{\cdot} 10^{-4}$     & $5.4{\cdot} 10^{-4}$     \\
		$3.4{\cdot} 10^{-4}$    & $2.7{\cdot} 10^{-4}$                & $5.7{\cdot} 10^{-2}$   & $1.6{\cdot} 10^{-2}$   & $4.2{\cdot} 10^{-3}$   & $1.1{\cdot} 10^{-3}$    & $3.8{\cdot} 10^{-4}$     & $2.7{\cdot} 10^{-4}$     \\ \bottomrule
	\end{tabular}
 \end{table}

\cref{tab:convergence_edge_modes_classical} shows the approximation error of the ACMS method and the corresponding errors for the nodal interpolant $I_h(u)$ for each mesh resolution.
We observe that the $H^1$-error decreases with a rate between $3.5$ and $4.0$, until it approaches the accuracy of the nodal interpolant $I_hu$ of the underlying finite element mesh. Similarly, the $L^2$-error decays cubically for sufficiently fine finite element meshes.
The $L^2$-convergence is in good agreement with~\cref{thm:convergence_interface_L2} and the relation $\lambda_i^e \sim i^2$ stated in~\cref{eq:edge_eigs_behavior}, while the $H^1$-convergence is better than predicted by \cref{lem:improved_rate_interface} in this example.
We may conclude that, already with a low number of edge modes $I_e\leq 32$, or which corresponds to $|S_\Gamma|\leq 384$, the ACMS solution achieves the accuracy of the nodal interpolant in the $H^1$-norm, which employs around $8{\cdot}10^6$ vertices for $h\approx 3.4{\cdot}10^{-4}$.
In fact, as shown in \cref{tab:convergence_edge_modes_classical_FEM}, the ACMS solution converges to the (standard) finite element solution $u_{\rm FEM}$ quadratically in the $H^1$-norm and cubically in the $L^2$-norm, respectively.
\begin{table}[t] 
	\centering\small\setlength\tabcolsep{0.55em}
	\caption{Classical Helmholtz example with $\kappa=1$: errors $e_{0,h}$ and $e_{1,h}$ as defined in~\cref{tab:errors} 
		for different number of edge modes $|S_\Gamma|\in\{24,28,\ldots,1536\}$, computed on a finite element mesh with $h\approx 3.4{\cdot}10^{-4}$ and $8\,327\,169$ vertices. \label{tab:convergence_edge_modes_classical_FEM}}
	\begin{tabular}{c c c c c c c c}
		\toprule
        & \multicolumn{7}{c}{$|S_\Gamma|$}\\
        \cline{2-8}
		& \multicolumn{1}{c}{24} & \multicolumn{1}{c}{48} & \multicolumn{1}{c}{96} & \multicolumn{1}{c}{192} & \multicolumn{1}{c}{384} & \multicolumn{1}{c}{768} & \multicolumn{1}{c}{1536} \\ \hline
		       $e_{0,h}$        & $4.1{\cdot} 10^{-3}$   & $7.5{\cdot} 10^{-4}$   & $1.1{\cdot} 10^{-4}$   & $1.5{\cdot} 10^{-5}$    & $1.9{\cdot} 10^{-6}$    & $2.5{\cdot} 10^{-7}$    & $3.1{\cdot} 10^{-8}$     \\
		       $e_{1,h}$        & $5.7{\cdot} 10^{-2}$   & $1.6{\cdot} 10^{-2}$   & $4.2{\cdot} 10^{-3}$   & $1.0{\cdot} 10^{-3}$    & $2.7{\cdot} 10^{-4}$    & $6.9{\cdot} 10^{-5}$    & $1.7{\cdot} 10^{-5}$     \\ \bottomrule
	\end{tabular}
\end{table}





\subsubsection{Higher wavenumbers}
We now repeat the previous numerical experiments focusing on the effect of an increasing wavenumber; namely, we test our method for $\kappa=2,4,8,16,32,64,128$.
Let us mention that~\cref{ass:eig_not_zero} is satisfied for all chosen wavenumbers and that the plane wave solution $u$ changes with $\kappa$.

In \Cref{tab:convergence_edge_modes_wavenumbers}, we show the convergence of the ACMS method to the plane wave solution and, for comparison, we display the error between the nodal interpolant and the finite element solution.
For lower wavenumber $\kappa=2$, we observe a similar behavior as in the previous section, i.e., close to cubic convergence of $e_{0}$ and close to quadratic convergence of $e_{1}$ until it occurs a saturation due to the limited resolution of the underlying FEM mesh. 

For wavenumbers $\kappa\leq 64$, we observe that $e_{1,{\rm int}}$ and $e_{1,{\rm FEM}}$ behave very similarly when $\kappa$ is increased.
However, if $\kappa=128$, $e_{1,{\rm FEM}}$ is more than one order of magnitude larger than $e_{1,{\rm int}}$, which may indicate that the FEM mesh is too coarse.
For $|S_\Gamma|\geq 768$, i.e., at least $64$ modes per edge, the ACMS error $e_1$ is close to the FEM error $e_{1,{\rm FEM}}$ for all $\kappa$. We also observe that, for $\kappa\geq 32$, $|S_\Gamma|$ has to be sufficiently large to have monotonic decay in $e_1$. 
Notably, up to a certain number of edge modes, we may even see an increase in $e_1$. Ultimately, after reaching a certain threshold in the number of edge modes, which increases with $\kappa$, we observe a significant drop in the error $e_1$, bringing it to a comparable level as $e_{1,{\rm FEM}}$.

The convergence of the ACMS solution to the FEM solution is also verified in  \Cref{tab:convergence_edge_modes_L2_wavenumbers}, where the corresponding errors are shown for different wavenumbers.
We observe a similar convergence behavior as for $e_1$ and $e_0$, respectively, without a saturation effect.
If one is satisfied with the approximation errors achieved by $u_{\rm FEM}$, we may again conclude that the ACMS method can yield good approximations already with a moderate number of degrees of freedom. We note that, by using higher order elements or further mesh refinements, the accuracy of $u_{\rm FEM}$ may be increased; see, e.g.,~\cite{MelenkSauter2010, MelenkSauter2011}. Then, we would expect that the corresponding solution of the ACMS method would also show better accuracy in approximating the exact solution. We will investigate this in future work.

\begin{table}[t]
	\centering\small\setlength\tabcolsep{0.5em}
	\caption{Classical Helmholtz example with higher wavenumbers: errors $e_{0}$ (top) and $e_{1}$ (bottom) as defined in~\cref{tab:errors} for different number of edge modes $|S_\Gamma|\in\{48,96,\ldots,1536\}$ and wavenumbers $\kappa$, computed on a finite element mesh with $h\approx 3.4{\cdot}10^{-4}$ and $8\,327\,169$ vertices. For comparison, we display the interpolation error for the nodal interpolant as well as the FEM approximation error $u_{\rm FEM}-u$ in the $H^1$- and $L^2$-norm, respectively.\label{tab:convergence_edge_modes_wavenumbers}}
	\begin{tabular}{r l l l l l l l l }
		\toprule
		         &                                     &                                     &                                                             \multicolumn{6}{c}{$|S_\Gamma|$}                                                             \\ \cline{4-9}
		$\kappa$ & \multicolumn{1}{c}{$e_{0,\rm int}$} & \multicolumn{1}{c}{$e_{0,\rm FEM}$} & \multicolumn{1}{c}{48} & \multicolumn{1}{c}{96} & \multicolumn{1}{c}{192} & \multicolumn{1}{c}{384} & \multicolumn{1}{c}{768} & \multicolumn{1}{c}{1536} \\ \hline
		       2 & $1.9{\cdot} 10^{-7}$                & $4.4{\cdot} 10^{-7}$                & $2.2{\cdot}10^{-3}$    & $3.5{\cdot}10^{-4}$    & $4.9{\cdot}10^{-5}$     & $6.6{\cdot}10^{-6}$     & $9.7{\cdot}10^{-7}$     & $4.5{\cdot}10^{-7}$      \\
		       4 & $7.6{\cdot} 10^{-7}$                & $3.5{\cdot} 10^{-6}$                & $1.3{\cdot}10^{-2}$    & $1.6{\cdot}10^{-3}$    & $2.1{\cdot}10^{-4}$     & $2.7{\cdot}10^{-5}$     & $4.9{\cdot}10^{-6}$     & $3.5{\cdot}10^{-6}$      \\
		       8 & $3.0{\cdot} 10^{-6}$                & $4.2{\cdot} 10^{-5}$                & $1.5{\cdot}10^{-1}$    & $7.8{\cdot}10^{-3}$    & $8.5{\cdot}10^{-4}$     & $1.1{\cdot}10^{-4}$     & $4.5{\cdot}10^{-5}$     & $4.2{\cdot}10^{-5}$      \\
		      16 & $1.2{\cdot} 10^{-5}$                & $4.2{\cdot} 10^{-4}$                & $1.2{\cdot}10^{ 0}$    & $1.6{\cdot}10^{-1}$    & $6.0{\cdot}10^{-3}$     & $7.2{\cdot}10^{-4}$     & $4.3{\cdot}10^{-4}$     & $4.2{\cdot}10^{-4}$      \\
		      32 & $4.8{\cdot} 10^{-5}$                & $4.5{\cdot} 10^{-3}$                & $4.9{\cdot}10^{0}$     & $2.2{\cdot}10^{0}$     & $3.5{\cdot}10^{-1}$     & $1.0{\cdot}10^{-2}$     & $4.7{\cdot}10^{-3}$     & $4.5{\cdot}10^{-3}$      \\
		      64 & $1.9{\cdot} 10^{-4}$                & $1.7{\cdot} 10^{-2}$                & $2.4{\cdot}10^{0 }$    & $2.0{\cdot}10^{0 }$    & $2.7{\cdot}10^{0 }$     & $8.7{\cdot}10^{-2}$     & $1.8{\cdot}10^{-2}$     & $1.7{\cdot}10^{-2}$      \\
		     128 & $7.7{\cdot} 10^{-4}$                & $4.3{\cdot} 10^{-1}$                & $3.4{\cdot}10^{ 0}$    & $2.0{\cdot}10^{ 0}$    & $4.4{\cdot}10^{ 0}$     & $9.9{\cdot}10^{ 0}$     & $4.8{\cdot}10^{-1}$     & $4.3{\cdot}10^{-1}$      \\ \bottomrule\toprule
		$\kappa$ & \multicolumn{1}{c}{$e_{1,\rm int}$} & \multicolumn{1}{c}{$e_{1,\rm FEM}$} & \multicolumn{1}{c}{48} & \multicolumn{1}{c}{96} & \multicolumn{1}{c}{192} & \multicolumn{1}{c}{384} & \multicolumn{1}{c}{768} & \multicolumn{1}{c}{1536} \\ \hline
		       2 & $1.0{\cdot} 10^{-3}$                & $1.0{\cdot} 10^{-3}$                & $4.9{\cdot}10^{-2}$    & $1.3{\cdot}10^{-2}$    & $3.8{\cdot}10^{-3}$     & $1.4{\cdot}10^{-3}$     & $1.1{\cdot}10^{-3}$     & $1.0{\cdot}10^{-3}$      \\
		       4 & $4.3{\cdot} 10^{-3}$                & $4.3{\cdot} 10^{-3}$                & $2.2{\cdot}10^{-1}$    & $5.6{\cdot}10^{-2}$    & $1.4{\cdot}10^{-2}$     & $5.6{\cdot}10^{-3}$     & $4.4{\cdot}10^{-3}$     & $4.3{\cdot}10^{-3}$      \\
		       8 & $1.7{\cdot} 10^{-2}$                & $1.7{\cdot} 10^{-2}$                & $1.8{\cdot}10^{0 }$    & $2.3{\cdot}10^{-1}$    & $5.8{\cdot}10^{-2}$     & $2.2{\cdot}10^{-2}$     & $1.7{\cdot}10^{-2}$     & $1.7{\cdot}10^{-2}$      \\
		      16 & $6.9{\cdot} 10^{-2}$                & $6.9{\cdot} 10^{-2}$                & $1.9{\cdot}10^{1 }$    & $3.0{\cdot}10^{0 }$    & $2.5{\cdot}10^{-1}$     & $8.9{\cdot}10^{-2}$     & $7.1{\cdot}10^{-2}$     & $6.9{\cdot}10^{-2}$      \\
		      32 & $2.7{\cdot} 10^{-1}$                & $3.1{\cdot} 10^{-1}$                & $1.5{\cdot}10^{2}$     & $7.2{\cdot}10^{1}$     & $1.1{\cdot}10^{1}$      & $4.9{\cdot}10^{-1}$     & $3.2{\cdot}10^{-1}$     & $3.1{\cdot}10^{-1}$      \\
		      64 & $1.1{\cdot} 10^{0}$                 & $1.5{\cdot} 10^{0}$                 & $1.5{\cdot}10^{2}$     & $1.3{\cdot}10^{2}$     & $1.7{\cdot}10^{2}$      & $6.1{\cdot}10^{0}$      & $1.6{\cdot}10^{0}$      & $1.5{\cdot}10^{0}$       \\
		     128 & $4.4{\cdot} 10^{0}$                 & $5.5{\cdot} 10^{1}$                 & $4.4{\cdot}10^{2}$     & $2.5{\cdot}10^{2}$     & $5.7{\cdot}10^{2}$      & $1.2{\cdot}10^{3}$      & $6.2{\cdot}10^{1}$      & $5.5{\cdot}10^{1}$       \\ \bottomrule
	\end{tabular}
\end{table}

\begin{table}[t]
	\centering\small\setlength\tabcolsep{0.5em}
	\caption{Classical Helmholtz example with higher wavenumbers: errors $e_{0,h}$ (top) and $e_{1,h}$ (bottom) as defined in~\cref{tab:errors} for different number of edge modes $|S_\Gamma|\in\{48,96,\ldots,1536\}$ and wavenumbers $\kappa$, computed on a finite element mesh with $h\approx 3.4{\cdot}10^{-4}$ and $8\,327\,169$ vertices. For comparison, we display the interpolation error for the nodal interpolant as well as the FEM approximation error $u_{\rm FEM}-u$ in the $H^1$- and $L^2$-norm, respectively. \label{tab:convergence_edge_modes_L2_wavenumbers}}
	\begin{tabular}{r l l l l l l l l}
		\toprule
      &&& \multicolumn{6}{c}{$|S_\Gamma|$}\\
        \cline{4-9}
		$\kappa$ & \multicolumn{1}{c}{$e_{0,\rm int}$} & \multicolumn{1}{c}{$e_{0,\rm FEM}$} & \multicolumn{1}{c}{48} & \multicolumn{1}{c}{96} & \multicolumn{1}{c}{192} & \multicolumn{1}{c}{384} & \multicolumn{1}{c}{768} & \multicolumn{1}{c}{1536} \\ \hline
		       2 & $1.9{\cdot} 10^{-7}$                & $4.4{\cdot} 10^{-7}$                & $2.2{\cdot} 10^{-3}$   & $3.5{\cdot} 10^{-4}$   & $4.9{\cdot} 10^{-5}$    & $6.6{\cdot} 10^{-6}$    & $8.5{\cdot} 10^{-7}$    & $1.0{\cdot} 10^{-7}$     \\
		       4 & $7.6{\cdot} 10^{-7}$                & $3.5{\cdot} 10^{-6}$                & $1.3{\cdot} 10^{-2}$   & $1.6{\cdot} 10^{-3}$   & $2.1{\cdot} 10^{-4}$    & $2.6{\cdot} 10^{-5}$    & $3.3{\cdot} 10^{-6}$    & $4.2{\cdot} 10^{-7}$     \\
		       8 & $3.0{\cdot} 10^{-6}$                & $4.2{\cdot} 10^{-5}$                & $1.5{\cdot} 10^{-1}$   & $7.8{\cdot} 10^{-3}$   & $8.5{\cdot} 10^{-4}$    & $1.0{\cdot} 10^{-4}$    & $1.3{\cdot} 10^{-5}$    & $1.6{\cdot} 10^{-6}$     \\
		      16 & $1.2{\cdot} 10^{-5}$                & $4.2{\cdot} 10^{-4}$                & $1.2{\cdot} 10^{0}$    & $1.6{\cdot} 10^{-1}$   & $5.8{\cdot} 10^{-3}$    & $5.0{\cdot} 10^{-4}$    & $5.4{\cdot} 10^{-5}$    & $6.6{\cdot} 10^{-6}$     \\
		      32 & $4.8{\cdot} 10^{-5}$                & $4.5{\cdot} 10^{-3}$                & $4.9{\cdot} 10^{0}$    & $2.2{\cdot} 10^{0}$    & $3.4{\cdot} 10^{-1}$    & $7.2{\cdot} 10^{-3}$    & $4.3{\cdot} 10^{-4}$    & $3.4{\cdot} 10^{-5}$     \\
		      64 & $1.9{\cdot} 10^{-4}$                & $1.7{\cdot} 10^{-2}$                & $2.4{\cdot} 10^{0}$    & $2.0{\cdot} 10^{0}$    & $2.7{\cdot} 10^{0}$     & $8.0{\cdot} 10^{-2}$    & $2.2{\cdot} 10^{-3}$    & $1.5{\cdot} 10^{-4}$     \\
		     128 & $7.7{\cdot} 10^{-4}$                & $4.3{\cdot} 10^{-1}$                & $3.5{\cdot} 10^{0}$    & $2.0{\cdot} 10^{0}$    & $4.5{\cdot} 10^{0}$     & $9.9{\cdot} 10^{0}$     & $2.4{\cdot} 10^{-1}$    & $4.1{\cdot} 10^{-3}$     \\ \bottomrule\toprule
		$\kappa$ & \multicolumn{1}{c}{$e_{1,\rm int}$} & \multicolumn{1}{c}{$e_{1,\rm FEM}$} & \multicolumn{1}{c}{48} & \multicolumn{1}{c}{96} & \multicolumn{1}{c}{192} & \multicolumn{1}{c}{384} & \multicolumn{1}{c}{768} & \multicolumn{1}{c}{1536} \\ \hline
		       2 & $1.0{\cdot} 10^{-3}$                & $1.0{\cdot} 10^{-3}$                & $4.9{\cdot} 10^{-2}$   & $1.3{\cdot} 10^{-2}$   & $3.6{\cdot} 10^{-3}$    & $9.4{\cdot} 10^{-4}$    & $2.4{\cdot} 10^{-4}$    & $6.0{\cdot} 10^{-5}$     \\
		       4 & $4.3{\cdot} 10^{-3}$                & $4.3{\cdot} 10^{-3}$                & $2.2{\cdot} 10^{-1}$   & $5.6{\cdot} 10^{-2}$   & $1.4{\cdot} 10^{-2}$    & $3.6{\cdot} 10^{-3}$    & $9.0{\cdot} 10^{-4}$    & $2.2{\cdot} 10^{-4}$     \\
		       8 & $1.7{\cdot} 10^{-2}$                & $1.7{\cdot} 10^{-2}$                & $1.8{\cdot} 10^{0}$    & $2.3{\cdot} 10^{-1}$   & $5.6{\cdot} 10^{-2}$    & $1.4{\cdot} 10^{-2}$    & $3.5{\cdot} 10^{-3}$    & $8.9{\cdot} 10^{-4}$     \\
		      16 & $6.9{\cdot} 10^{-2}$                & $6.9{\cdot} 10^{-2}$                & $1.9{\cdot} 10^{1}$    & $3.0{\cdot} 10^{0}$    & $2.4{\cdot} 10^{-1}$    & $5.5{\cdot} 10^{-2}$    & $1.4{\cdot} 10^{-2}$    & $3.5{\cdot} 10^{-3}$     \\
		      32 & $2.7{\cdot} 10^{-1}$                & $3.1{\cdot} 10^{-1}$                & $1.5{\cdot} 10^{2}$    & $7.2{\cdot} 10^{1}$    & $1.1{\cdot} 10^{1}$     & $3.4{\cdot} 10^{-1}$    & $5.9{\cdot} 10^{-2}$    & $1.4{\cdot} 10^{-2}$     \\
		      64 & $1.1{\cdot} 10^{0}$                 & $1.5{\cdot} 10^{0}$                 & $1.5{\cdot} 10^{2}$    & $1.3{\cdot} 10^{2}$    & $1.7{\cdot} 10^{2}$     & $5.6{\cdot} 10^{0}$     & $2.8{\cdot} 10^{-1}$    & $5.8{\cdot} 10^{-2}$     \\
		     128 & $4.4{\cdot} 10^{0}$                 & $5.5{\cdot} 10^{1}$                 & $4.4{\cdot} 10^{2}$    & $2.6{\cdot} 10^{2}$    & $5.7{\cdot} 10^{2}$     & $1.2{\cdot} 10^{3}$     & $3.1{\cdot} 10^{1}$     & $5.9{\cdot} 10^{-1}$     \\ \bottomrule
	\end{tabular}
\end{table}

\subsection{Localized interior source}
In certain applications, such as in geophysics \cite{Wang2011},
the source terms of the wave propagation are localized. Therefore, let us study the behavior of the ACMS method for this type of setup. Let us consider $\Omega=B_1(0)$ with $\Gamma_R = \partial\Omega$, the coefficients $a(x)=1$ and $\beta(x) =1$, 
the non-zero source function
$f(x) = \exp(-200 | x- x_c|^2)$,
with $x_c=(\myfrac{1}{3},\myfrac{1}{3})$, $\kappa=1$, and $g \equiv 0$.
For this example, no analytical expression of the solution $u$ is available. Thus, we only investigate the convergence of the ACMS method towards the solution of the underlying finite element method, which is justified by the discussion in the previous sections.

\begin{table}[t]
	\centering\small \setlength\tabcolsep{0.55em}
	\caption{Example 5.2 with a localized interior source. Error 
		$e_{0,h}$ (top) and $e_{1,h}$ (bottom)
		as defined in~\cref{tab:errors} for different number of edge modes $|S_\Gamma|\in\{24,28,\ldots,768\}$ (in columns) and bubble functions $|S_B|$ (in rows) for $h=2.8{\cdot}10^{-3}$. 
		\label{tab:convergence_interior_source}}
	\begin{tabular}{r c c c c c c c}
		\toprule
        & \multicolumn{7}{c}{$|S_\Gamma|$}\\
        \cline{2-8}
		$|S_B|$ & \multicolumn{1}{c}{24} & \multicolumn{1}{c}{48} & \multicolumn{1}{c}{96} & \multicolumn{1}{c}{192} & \multicolumn{1}{c}{384} & \multicolumn{1}{c}{768} & \multicolumn{1}{c}{1536} \\ \hline
		      2 & $5.4{\cdot} 10^{-2}$   & $4.8{\cdot} 10^{-2}$   & $4.7{\cdot} 10^{-2}$   & $4.7{\cdot} 10^{-2}$    & $4.7{\cdot} 10^{-2}$    & $4.7{\cdot} 10^{-2}$    & $4.7{\cdot} 10^{-2}$     \\
		      4 & $2.9{\cdot} 10^{-2}$   & $2.6{\cdot} 10^{-2}$   & $2.6{\cdot} 10^{-2}$   & $2.6{\cdot} 10^{-2}$    & $2.6{\cdot} 10^{-2}$    & $2.6{\cdot} 10^{-2}$    & $2.6{\cdot} 10^{-2}$     \\
		      8 & $2.5{\cdot} 10^{-2}$   & $1.7{\cdot} 10^{-2}$   & $1.7{\cdot} 10^{-2}$   & $1.7{\cdot} 10^{-2}$    & $1.7{\cdot} 10^{-2}$    & $1.7{\cdot} 10^{-2}$    & $1.7{\cdot} 10^{-2}$     \\
		     16 & $2.1{\cdot} 10^{-2}$   & $9.9{\cdot} 10^{-3}$   & $9.5{\cdot} 10^{-3}$   & $9.5{\cdot} 10^{-3}$    & $9.5{\cdot} 10^{-3}$    & $9.5{\cdot} 10^{-3}$    & $9.5{\cdot} 10^{-3}$     \\
		     32 & $1.9{\cdot} 10^{-2}$   & $4.4{\cdot} 10^{-3}$   & $3.2{\cdot} 10^{-3}$   & $3.2{\cdot} 10^{-3}$    & $3.2{\cdot} 10^{-3}$    & $3.2{\cdot} 10^{-3}$    & $3.2{\cdot} 10^{-3}$     \\
		     64 & $1.9{\cdot} 10^{-2}$   & $2.8{\cdot} 10^{-3}$   & $6.2{\cdot} 10^{-4}$   & $6.1{\cdot} 10^{-4}$    & $6.1{\cdot} 10^{-4}$    & $6.1{\cdot} 10^{-4}$    & $6.1{\cdot} 10^{-4}$     \\
		    128 & $1.9{\cdot} 10^{-2}$   & $2.8{\cdot} 10^{-3}$   & $1.5{\cdot} 10^{-4}$   & $4.3{\cdot} 10^{-5}$    & $4.2{\cdot} 10^{-5}$    & $4.2{\cdot} 10^{-5}$    & $4.2{\cdot} 10^{-5}$     \\
		    256 & $1.9{\cdot} 10^{-2}$   & $2.8{\cdot} 10^{-3}$   & $1.4{\cdot} 10^{-4}$   & $1.0{\cdot} 10^{-5}$    & $1.4{\cdot} 10^{-6}$    & $3.6{\cdot} 10^{-7}$    & $3.1{\cdot} 10^{-7}$     \\
		    512 & $1.9{\cdot} 10^{-2}$   & $2.8{\cdot} 10^{-3}$   & $1.4{\cdot} 10^{-4}$   & $1.0{\cdot} 10^{-5}$    & $1.4{\cdot} 10^{-6}$    & $1.8{\cdot} 10^{-7}$    & $4.1{\cdot} 10^{-8}$     \\
		   1024 & $1.9{\cdot} 10^{-2}$   & $2.8{\cdot} 10^{-3}$   & $1.4{\cdot} 10^{-4}$   & $1.0{\cdot} 10^{-5}$    & $1.4{\cdot} 10^{-6}$    & $1.8{\cdot} 10^{-7}$    & $3.1{\cdot} 10^{-8}$     \\ \bottomrule\toprule
		$|S_B|$ & \multicolumn{1}{c}{24} & \multicolumn{1}{c}{48} & \multicolumn{1}{c}{96} & \multicolumn{1}{c}{192} & \multicolumn{1}{c}{384} & \multicolumn{1}{c}{768} & \multicolumn{1}{c}{1536} \\ \hline
		      2 & $3.6{\cdot} 10^{-1}$   & $3.4{\cdot} 10^{-1}$   & $3.4{\cdot} 10^{-1}$   & $3.4{\cdot} 10^{-1}$    & $3.4{\cdot} 10^{-1}$    & $3.4{\cdot} 10^{-1}$    & $3.4{\cdot} 10^{-1}$     \\
		      4 & $2.6{\cdot} 10^{-1}$   & $2.4{\cdot} 10^{-1}$   & $2.4{\cdot} 10^{-1}$   & $2.4{\cdot} 10^{-1}$    & $2.4{\cdot} 10^{-1}$    & $2.4{\cdot} 10^{-1}$    & $2.4{\cdot} 10^{-1}$     \\
		      8 & $2.1{\cdot} 10^{-1}$   & $1.9{\cdot} 10^{-1}$   & $1.8{\cdot} 10^{-1}$   & $1.8{\cdot} 10^{-1}$    & $1.8{\cdot} 10^{-1}$    & $1.8{\cdot} 10^{-1}$    & $1.8{\cdot} 10^{-1}$     \\
		     16 & $1.6{\cdot} 10^{-1}$   & $1.2{\cdot} 10^{-1}$   & $1.2{\cdot} 10^{-1}$   & $1.2{\cdot} 10^{-1}$    & $1.2{\cdot} 10^{-1}$    & $1.2{\cdot} 10^{-1}$    & $1.2{\cdot} 10^{-1}$     \\
		     32 & $1.1{\cdot} 10^{-1}$   & $5.9{\cdot} 10^{-2}$   & $5.4{\cdot} 10^{-2}$   & $5.4{\cdot} 10^{-2}$    & $5.4{\cdot} 10^{-2}$    & $5.4{\cdot} 10^{-2}$    & $5.4{\cdot} 10^{-2}$     \\
		     64 & $1.0{\cdot} 10^{-1}$   & $2.7{\cdot} 10^{-2}$   & $1.3{\cdot} 10^{-2}$   & $1.3{\cdot} 10^{-2}$    & $1.3{\cdot} 10^{-2}$    & $1.3{\cdot} 10^{-2}$    & $1.3{\cdot} 10^{-2}$     \\
		    128 & $1.0{\cdot} 10^{-1}$   & $2.3{\cdot} 10^{-2}$   & $2.5{\cdot} 10^{-3}$   & $1.2{\cdot} 10^{-3}$    & $1.2{\cdot} 10^{-3}$    & $1.2{\cdot} 10^{-3}$    & $1.2{\cdot} 10^{-3}$     \\
		    256 & $1.0{\cdot} 10^{-1}$   & $2.3{\cdot} 10^{-2}$   & $2.2{\cdot} 10^{-3}$   & $3.6{\cdot} 10^{-4}$    & $9.6{\cdot} 10^{-5}$    & $2.8{\cdot} 10^{-5}$    & $1.5{\cdot} 10^{-5}$     \\
		    512 & $1.0{\cdot} 10^{-1}$   & $2.3{\cdot} 10^{-2}$   & $2.2{\cdot} 10^{-3}$   & $3.6{\cdot} 10^{-4}$    & $9.6{\cdot} 10^{-5}$    & $2.5{\cdot} 10^{-5}$    & $8.1{\cdot} 10^{-6}$     \\
		   1024 & $1.0{\cdot} 10^{-1}$   & $2.3{\cdot} 10^{-2}$   & $2.2{\cdot} 10^{-3}$   & $3.6{\cdot} 10^{-4}$    & $9.5{\cdot} 10^{-5}$    & $2.5{\cdot} 10^{-5}$    & $7.9{\cdot} 10^{-6}$     \\ \bottomrule
	\end{tabular}
\end{table}

\cref{tab:convergence_interior_source} shows the results for varying numbers of bubble functions and edge modes. In view of~\cref{thm:bubble_convergence_rate}, we note that $\uB \approx 0$ outside the subdomain $\Oj$ with $x_c\in\Oj$. Hence, for an accurate approximation, we only need bubble functions in that subdomains $\Oj$.
Moreover, the linear system for the bubble component is diagonal and decoupled from the interface part. Consequently, the solution coefficient corresponding to a specific bubble basis function can be computed independently from any other ACMS basis function; cf.~\cref{eq:bubble_approx,eq:interface_approx}.

For this highly localized source, we observe very fast convergence to the finite element solution when $I_j\geq 32$, given a sufficiently high number of edge modes.
This can be explained by inspecting the proof of \cref{thm:bubble_convergence_rate}: using repeatedly~\cref{eq:bubbles}, integration-by-parts and that $f$ and all its derivative are negligible on $\partial\Oj$, we can estimate $(f, b_i^j)_{\Oj}$ by arbitrary powers of $\lambda_i^j$; see also~\cref{lem:improved_rate_interface} for similar steps. For a high number of bubble functions, a saturation effect occurs. This can be due to the limited quantity of edge modes or to the fact that the values of higher-order derivatives of $f$ are not negligible anymore.
When increasing the number of selected edge modes while keeping the number of bubble functions fixed, the convergence approaches the predicted quadratic and cubic rates for the $H^1$- and the $L^2$-error, respectively.
Since the Galerkin problems for the bubble part and the interface are decoupled (see~\cref{eq:bubble,eq:interface}), the size of the systems to be solved remains comparably small. 
Therefore, we conclude that, for a moderate number of degrees of freedom, we obtain very good approximations of the finite element solution.

\subsection{Localized boundary source} \label{sec:local_source}
Next, let us modify the classical Helmholtz example from \cref{sec:classical_Helmholtz_example} such that $\kappa=16$, $f=0$, and 
$g(x)=\exp(-200|x-x_c|^2)$, with $x_c=(-\myfrac{1}{\sqrt{2}},\myfrac{1}{\sqrt{2}})$. 
In~\cref{fig:solution_bdry_source}, we display the real and imaginary part of the finite element solution computed on a quasi-uniform mesh with $521\,217$ vertices, i.e., $h\approx 1.3\cdot 10^{-3}$. Although the source is very localized around $x_c$, the whole domain is excited, indicating the wave-type behavior of the solution.

\begin{figure}
    \centering
    \includegraphics[width=0.45\textwidth]{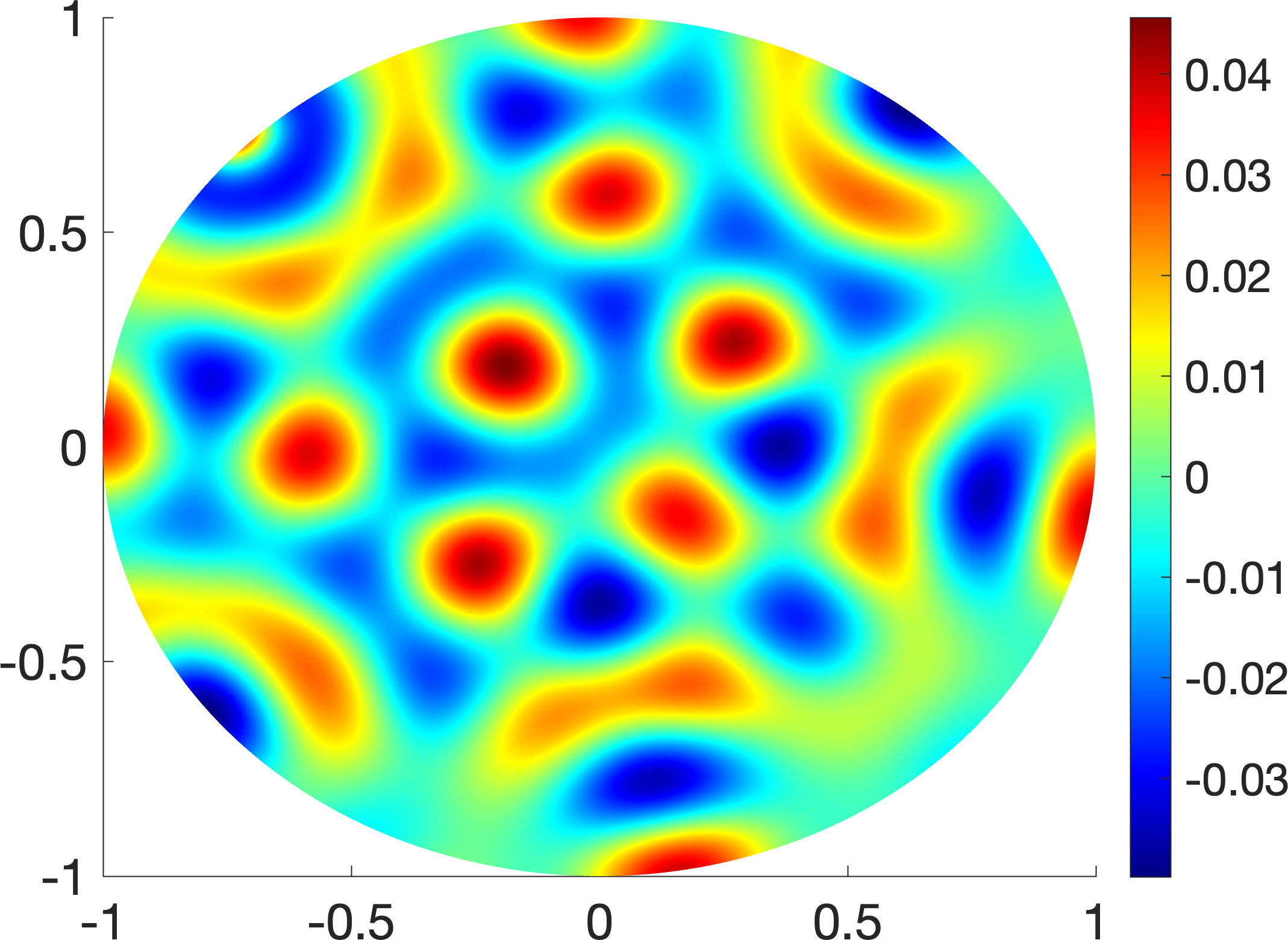}
	\hfill
    \includegraphics[width=0.45\textwidth]{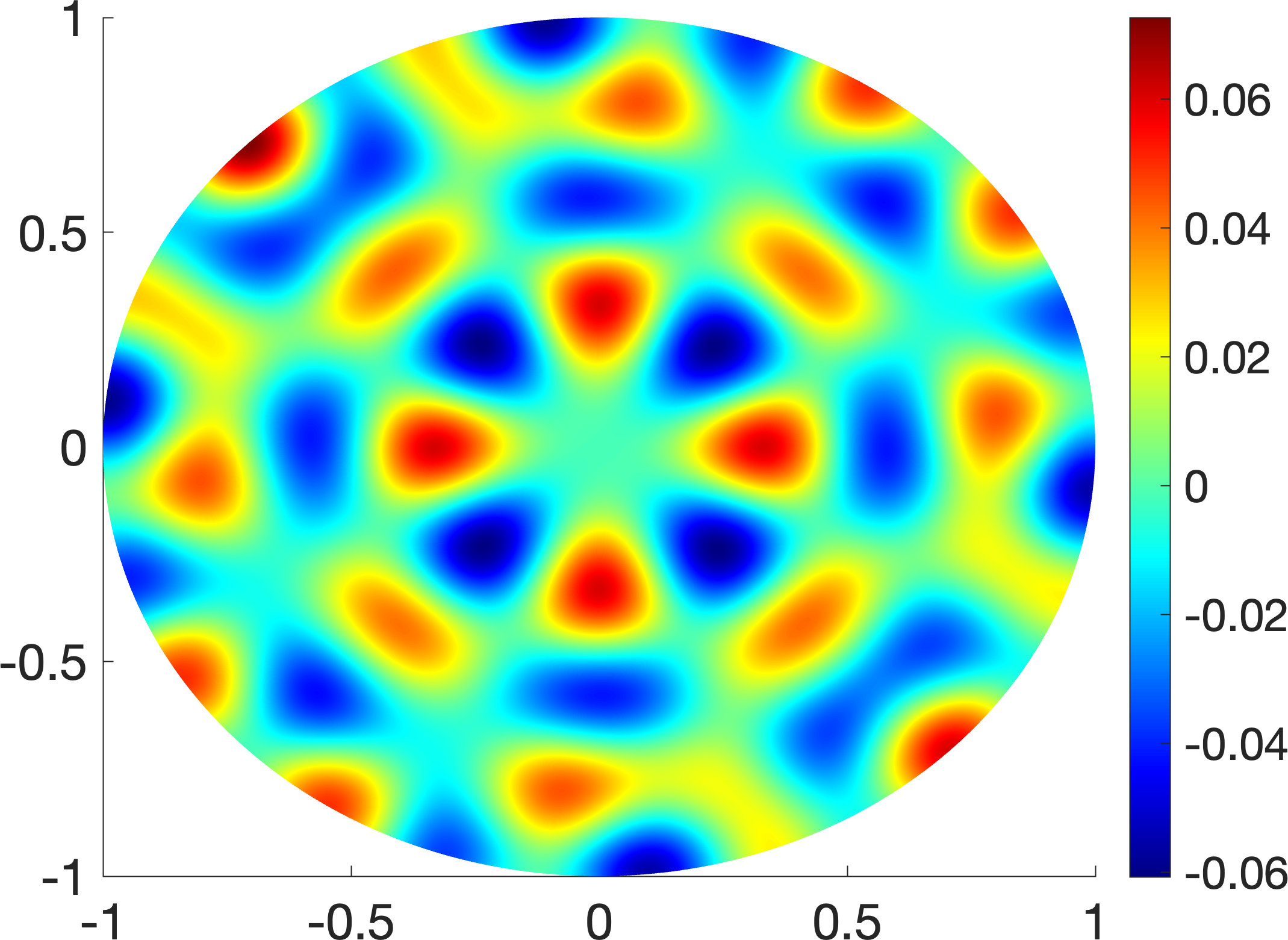}
    \caption{Finite element solution of the model problem with a localized boundary source from~\cref{sec:local_source}: real part (left) and imaginary part (right).
    \label{fig:solution_bdry_source}}
\end{figure}

\begin{table}[t]
	\centering\small \setlength\tabcolsep{0.55em}
	\caption{Example 5.4: Localized boundary source: relative errors $e_{0,h}^r$ and $e_{1,h}^r$ as defined in \cref{tab:errors} for different number of edge modes $|S_\Gamma|\in\{24,28,\ldots,3072\}$, computed on a grid with $h=1.3{\cdot} 10^{-3}$. \label{tab:convergence_edge_modes_loc_bdry_source}}
	\begin{tabular}{r r r r r r r r r}
		\toprule
        & \multicolumn{8}{c}{$|S_\Gamma|$}\\
        \cline{2-9}
		            & \multicolumn{1}{c}{24} & \multicolumn{1}{c}{48} & \multicolumn{1}{c}{96} & \multicolumn{1}{c}{192} & \multicolumn{1}{c}{384} & \multicolumn{1}{c}{768} & \multicolumn{1}{c}{1536} & \multicolumn{1}{c}{3072} \\ \hline
		$e_{0,h}^r$ & $1.6{\cdot} 10^{0}$    & $1.0{\cdot} 10^{0}$    & $1.7{\cdot} 10^{-1}$   & $7.3{\cdot} 10^{-3}$    & $4.1{\cdot} 10^{-4}$    & $3.3{\cdot} 10^{-5}$    & $3.5{\cdot} 10^{-6}$     & $4.3{\cdot} 10^{-7}$     \\
		$e_{1,h}^r$ & $1.6{\cdot} 10^{0}$    & $1.0{\cdot} 10^{0}$    & $1.8{\cdot} 10^{-1}$   & $1.4{\cdot} 10^{-2}$    & $1.9{\cdot} 10^{-3}$    & $4.6{\cdot} 10^{-4}$    & $1.1{\cdot} 10^{-4}$     & $2.8{\cdot} 10^{-5}$     \\ \bottomrule
	\end{tabular}
\end{table}


Since $f=0$ here, we do not need any bubble functions to obtain convergence according to \cref{rem:local_bubble}. As we do not have an analytic solution, we investigate the error between $u_{\rm FEM}$ and $\uh$ for varying number of edge modes.
The results are shown in~\cref{tab:convergence_edge_modes_loc_bdry_source}, where we consider the relative errors since the solution values are rather small.
As predicted by~\cref{thm:estimate_sesP_norm}, we observe second order convergence for the $H^1$-error, while the decay is initially slightly faster. The $L^2$-error shows a similar behavior with a convergence rate approaching third order; cf.~\cref{thm:convergence_interface_L2}.
We again may conclude that we can approximate the highly resolved standard finite element solution using a moderate number of edge modes in the ACMS method.

\subsection{Periodic structure} \label{sec:periodic}
Let us conclude our numerical experiments with a heterogeneous Helmholtz problem on the unit square $\Omega=[0,1]^2$ with a domain decomposition as in~\cref{fig:crystal} (left) and with a heterogeneous diffusion coefficient $a$ depicted in~\cref{fig:crystal} (right). This configuration is similar to the modeling of two-dimensional photonic crystals~\cite{Joannopoulos2008MoldingLight}. 
We choose $\beta(x) =1$, interior source $f=0$, and the localized boundary source on $\Gamma_R = \partial \Omega$, 
$g(x)=\exp(-\iota k\cdot x)\exp(-100|x-x_c|^2)$,
with $x_c=(0,\myfrac{1}{2})$, wave vector $k=\kappa (1,0)$ and wavenumber $\kappa=100$.
The corresponding finite element solution, which has been computed on a quasi-uniform triangulation with $5\,330\,337$ vertices, is shown in~\cref{fig:solution_crystal}.

\begin{figure}
    \centering
     \includegraphics[width=0.45\textwidth]{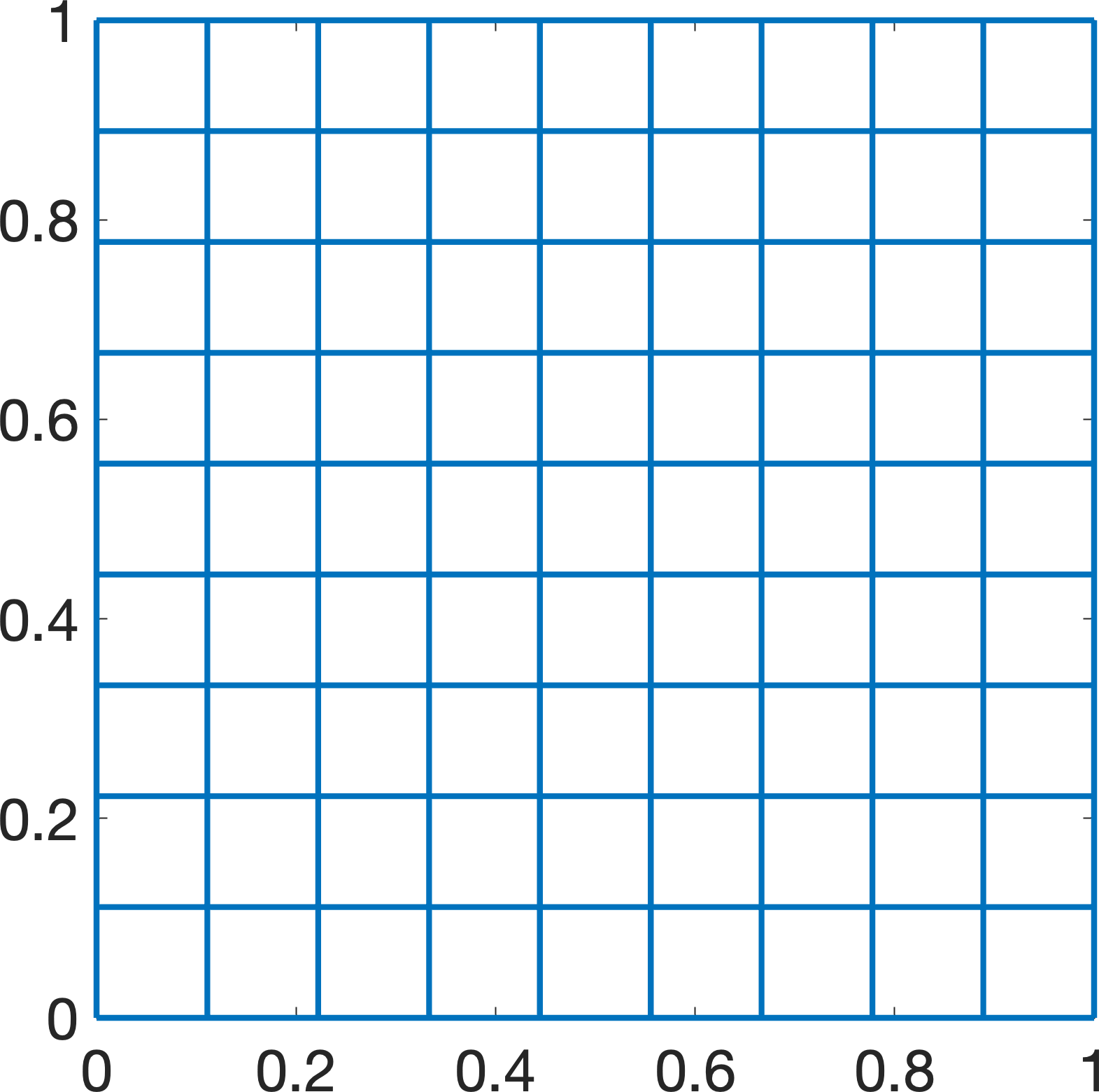}
	\hfill
   \includegraphics[width=0.45\textwidth]{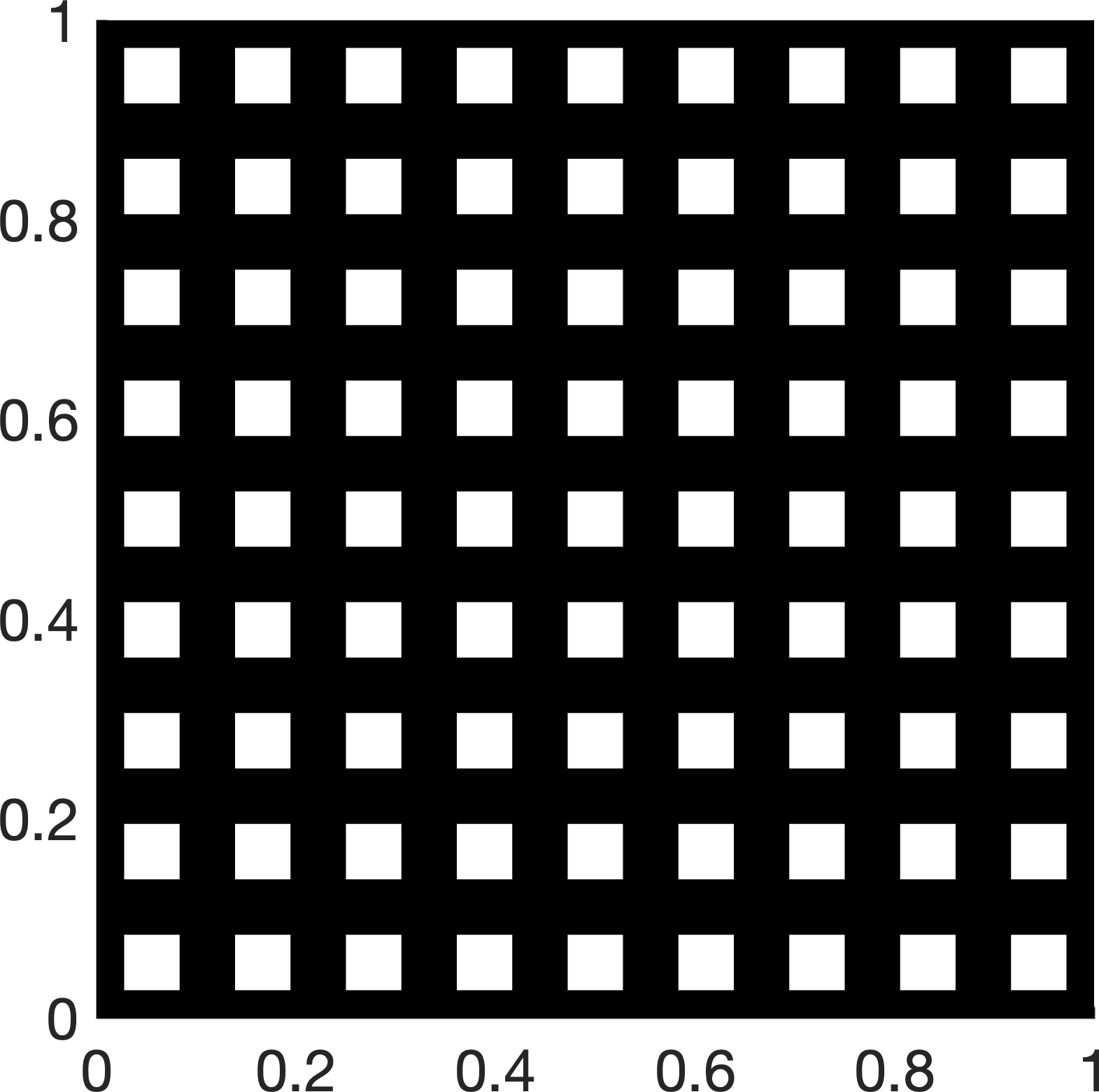}
    \caption{\textbf{Left:} unit square divided in $81$ subdomains for domain decomposition. The number of edges in $\E$ is $180$ and the number of vertices in $\V$ is $100$. \textbf{Right:} heterogeneous diffusion coefficient with $a=1$ in the black regions and $a=12$ in the white regions of the unit square. \label{fig:crystal}}
\end{figure}
\begin{figure}
    \centering
     \includegraphics[width=0.45\textwidth]{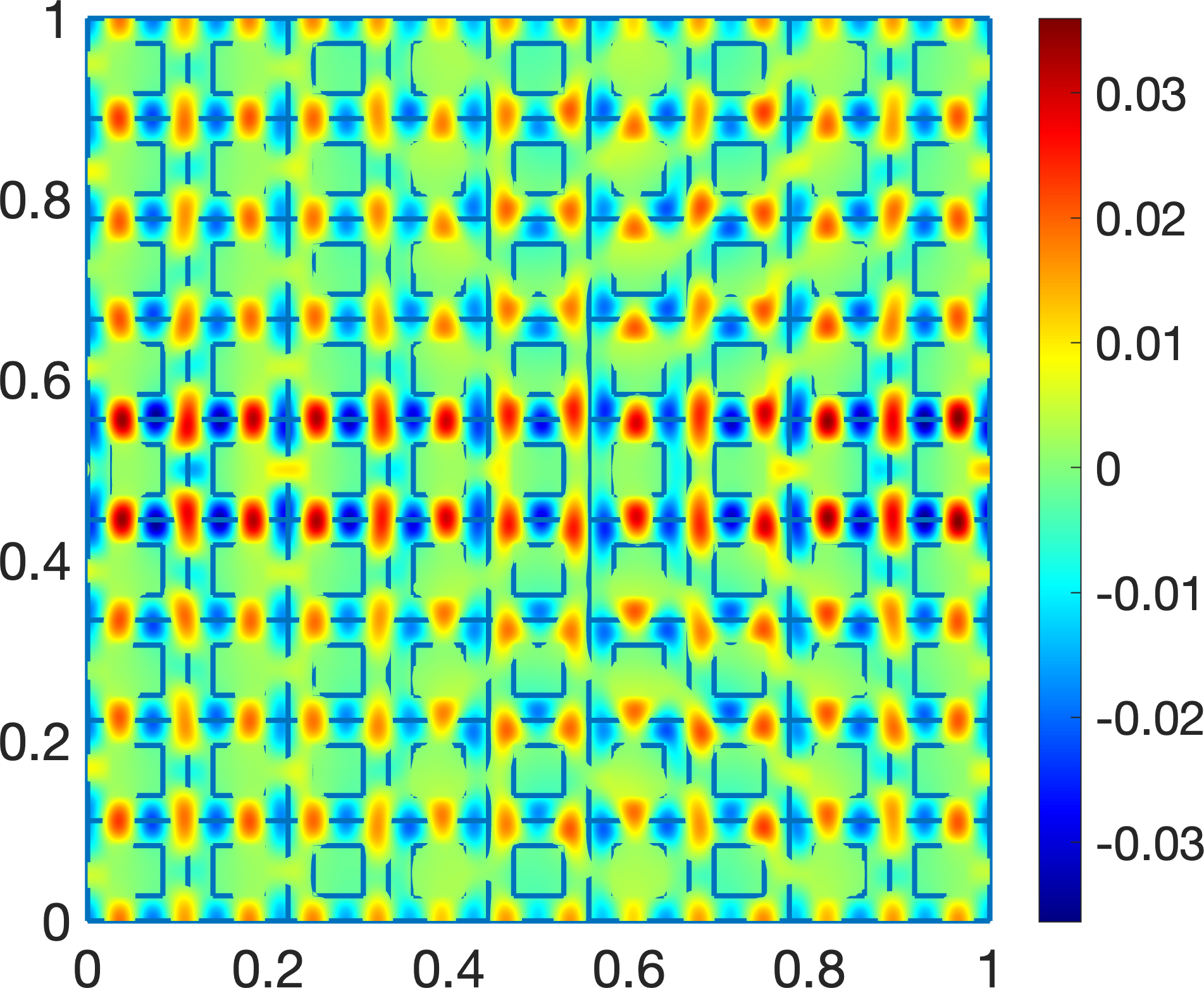}
     \hfill
    \includegraphics[width=0.45\textwidth]{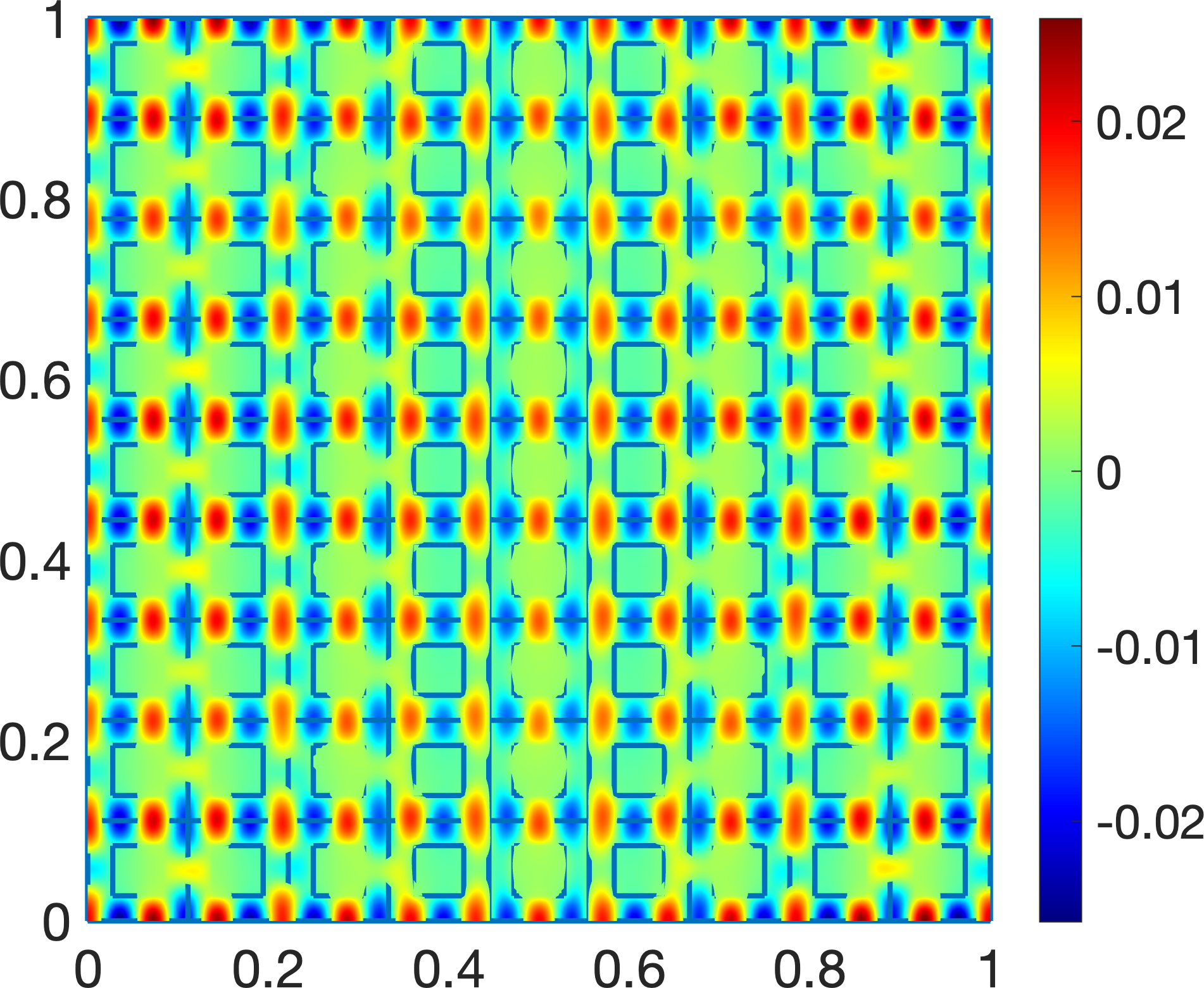}
    \caption{Finite element solution of the model problem with a localized boundary source from~\cref{sec:periodic}: real part (left) and imaginary part (right); the underlying structure of the coefficient $a$ is shown as well.
    \label{fig:solution_crystal}}
\end{figure}

\begin{table}[t]
	\centering\small \setlength\tabcolsep{0.55em}
	\caption{Example 5.4: Periodic structure. Relative errors $e_{0,h}^r$ and $e_{1,h}^r$ as defined in \cref{tab:errors} for different number of edge modes $|S_\Gamma|$ computed on a grid with $5\,330\,337$ vertices. \label{tab:convergence_periodic}}
	\begin{tabular}{r r r r r r r}
		\toprule
    & \multicolumn{6}{c}{$|S_\Gamma|$}\\
        \cline{2-7}
		            & \multicolumn{1}{c}{360} & \multicolumn{1}{c}{720} & \multicolumn{1}{c}{1\,440} & \multicolumn{1}{c}{2\,880} & \multicolumn{1}{c}{5\,760} & \multicolumn{1}{c}{11\,520} \\ 
\hline
		$e_{0,h}^r$ & $1.1{\cdot} 10^{0}$     & $1.3{\cdot} 10^{0}$     & $1.6{\cdot} 10^{-1}$     & $1.0{\cdot} 10^{-2}$     & $7.5{\cdot} 10^{-4}$     & $5.0{\cdot} 10^{-5}$        \\
		$e_{1,h}^r$ & $1.1{\cdot} 10^{0}$     & $1.3{\cdot} 10^{0}$     & $1.6{\cdot} 10^{-1}$     & $1.1{\cdot} 10^{-2}$     & $1.3{\cdot} 10^{-3}$     & $2.9{\cdot} 10^{-4}$        \\ \bottomrule
	\end{tabular}
\end{table}

Since, again, $f=0$, we only need to study the behavior for varying numbers of edge modes, while the bubble part of the solution $\uB$ is zero.
Note that we can choose $\Omega_\Gamma$ defined in \cref{eq:Omega_G} inside the region with $a=1$; cf.~\cref{fig:crystal}.
The relative $H^1$- and $L^2$-errors between the finite element solution and the ACMS solution are listed in~\cref{tab:convergence_periodic}. For $|S_\Gamma|=2\,880$, which corresponds to the case of $16$ modes per edge, the ACMS discretization already yields a good approximation to the highly resolved finite element solution. Moreover, the convergence of the error is initially even better than predicted by our theoretical results for $|S_\Gamma| \geq 720$.

\section*{Conclusions}
We extended the ACMS method, which has originally been developed for elliptic problems, to the heterogeneous Helmholtz equation.
This framework is based on a decomposition into local Helmholtz problems and an interface problem, that can be solved separately. The numerical approximation of those problems is achieved by using basis functions with local support and which can be constructed locally as well.

The error analysis developed for the investigated method is based on the abstract framework introduced in~\cite{GrahamSauter2019}.
We proved error estimates in the $H^1$- and the $L^2$-norm, and we obtained algebraic decay in the number of basis functions (modes) being used. 
In order to be able to apply the framework of~\cite{GrahamSauter2019}, we presented estimates for the adjoint approximability constants, showing that the number of edgemodes should scale essentially like $\omega^2$.
Moreover, we showed applicability of the method also if the diffusion coefficient is not smooth, as long as smoothness in a neighborhood of the interface is guaranteed.

Finally, we exemplified the theoretical error bounds by numerical experiments, which show the accuracy of the inspected ACMS method for moderate wavenumbers. We constructed the ACMS basis functions using a finite element method, and we observed that the resulting discrete ACMS method approximates well the corresponding finite element solution. Thus, we may conclude that the accuracy of the discrete ACMS method depends on the accuracy of the underlying finite element method, and smaller errors might be achieved by using high-order finite elements. 
\section*{Acknowledgments}
EG and MS acknowledge support by the Dutch Research council (NWO) via grant OCENW.GROOT.2019.071.
AH and MS acknowledge support by the 4TU.AMI SRI \textit{Bridging Numerical Analysis and Machine Learning}.

\bibliographystyle{siamplain}
\bibliography{references}

\end{document}